\documentclass[a4paper,10pt,fleqn]{article}
\usepackage{amsmath}
\usepackage{amssymb}
\usepackage{theorem}
\usepackage{euscript}
\newtheorem{theorem}{Theorem}[section]
\newtheorem{lemma}{Lemma}[section]
\newtheorem{corollary}{Corollary}[section]
\numberwithin{equation}{section}
\newcommand{\bj}{{\bf j}}
\newcommand{\bk}{{\bf k}}
\newcommand{\bh}{{\bf h}}

\newcommand{\bm}{{\bf m}}
\newcommand{\bn}{{\bf n}}

\newcommand{\br}{{\bf r}}
\newcommand{\bs}{{\bf s}}
\newcommand{\bt}{{\bf t}}
\newcommand{\bx}{{\bf x}}
\newcommand{\bX}{{\bf X}}
\newcommand{\by}{{\bf y}}
\newcommand{\bY}{{\bf Y}}
\newcommand{\bz}{{\bf z}}
\newcommand{\bM}{{\bf M}}
\newcommand{\rd}{{\rm d}}
\def\II{{\mathbb I}}

\def\ZZ{{\mathbb Z}}
\def\NN{{\mathbb N}}
\def\RR{{\mathbb R}}

\def\Pp{{\mathcal P}}

\def\Ff{{\mathcal F}}
\def\TT{{\mathbb T}}
\def\IId{{\mathbb I}^d}

\def\NNd{{\mathbb N}^d}
\def\RRd{{\mathbb R}^d}
\def\TTd{{\mathbb T}^d}

\def\ZZd{{\mathbb Z}^d}

\def\ZZdp{{\mathbb Z}^d_+}

\def\ZZdpe{{\mathbb Z}^d_+(e)}
\def\ZZdpu{{\mathbb Z}^d_+(u)}

\def\ZZdu{{\mathbb Z}^d(u)}
\def\Pp{{\mathcal P}}

\def\Ff{{\mathcal F}}

\def\Wrp{W^r_p}

\def\Urp{U^r_p}
\def\supp{\operatorname{supp}}
\title{\sffamily B-spline quasi-interpolation sampling representation 
and  sampling recovery in Sobolev spaces of mixed smoothness 
\author{ 
Dinh D\~ung \\[5mm]
Information Technology Institute, Vietnam National University \\
144 Xuan Thuy, Cau Giay, Hanoi, Vietnam\\
{\ttfamily dinhzung@gmail.com}\\[4mm]}}
\date{\ttfamily November 27, 2016 --  Version 0.9}
\begin{document}
\maketitle

\begin{abstract}
We proved direct and inverse theorems on B-spline quasi-interpolation sampling representation with a Littlewood-Paley-type norm equivalence in Sobolev spaces $W^r_p$ of mixed smoothness $r$, established estimates of the approximation error of recovery in $L_q$-norm of functions from the unit ball $U^r_p$ in the spaces  $W^r_p$ by linear sampling algorithms based on this representation, the asymptotic optimality of these sampling algorithms in terms of Smolyak sampling width $r^s_n(U^r_p, L_q)$ and sampling width $r_n(U^r_p, L_q)$.


  
\end{abstract}

\section{Introduction}

The purpose of the present paper is to study B-spline quasi-interpolation sampling representations with Littlewood-Paley-type norm equivalence in Sobolev spaces of a mixed smoothness,  linear sampling algorithms on Smolyak grids based on them for recovery of functions from these spaces, and the optimality of these algorithms. Let us first briefly describe some main results. The historical comments will be given after.

We are interested in quasi-interpolation sampling representation and sampling recovery of functions on $\RRd$ which are $1$-periodic at each variable. It is convenient to consider them as functions defined in the $d$-torus $\TTd = [0,1]^d$ which is defined as the cross product of
$d$ copies of the interval $[0,1]$ with the identification of the end points. To avoid confusion, we use the notation $\IId$ to denote the standard unit $d$-cube $[0,1]^d$. 

For $0 < p \le \infty$ and $r > 0$, denote by $\Wrp$ the Sobolev-type space of functions on $\TTd$ having uniform mixed smoothness $r$. 
If $1 < p < \infty$ and $\max(\frac{1}{p},\frac{1}{2}) < r < \ell - 1$, then we prove that
every function $f \in \Wrp$ can be represented as B-spline series 
\begin{equation} \label{introduction[B-splineRepresentation]}
f \ = \sum_{\bk \in {\ZZ}^d_+} \ q_\bk(f) = 
\sum_{\bk \in {\ZZ}^d_+} \sum_{\bs \in I(\bk)} c_{\bk,\bs}(f)N_{\bk,\bs}, 
\end{equation}  
converging in the norm of $\Wrp$, where the coefficient functionals $c_{\bk,\bs}(f)$ are explicitly constructed as
linear combinations of at most $m_0$ of function
values of $f$ for some $m_0 \in \NN$ which is independent of $\bk,\bs$ and $f$, $N_{\bk,\bs}$ are the tensor product of 
integer translates of dyadic scaled periodic B-splines of even order $\ell$ (see Subsection \ref{B-spline Q-I} for details), and 
\begin{equation}  \label{I(bk)}
I(\bk) := \{\bs \in \ZZdp: s_j = 0,1,..., 2^{k_j} - 1, \ j \in [d]\}.
\end{equation}
Moreover, for this representation there holds the norm equivalence
\begin{equation} \label{introduction[normeqWrp]}
\biggl\| \biggl(\sum_{\bk \in \ZZdp} 
\Big|2^{r|\bk|_1} q_\bk(f)\Big|^2 \biggl)^{1/2}\biggl\|_p
\ \asymp \
\|f\|_{\Wrp},
\quad  \forall f \in \Wrp. 
\end{equation}

Let $\bX_n = \{\bx^j\}_{j=1}^n$ be a set of $n$ points 
in $\TTd$, $\Phi_n = \{\varphi_j\}_{j =1}^n$ a family of  
$n$ functions on ${\II}^d$. If  $f$ is a function on $\TTd$, for approximately recovering $f$ from the sampled values $f(\bx^1),..., f(\bx^n)$, we  define the linear sampling algorithm $L_n(X_n,\Phi_n,\cdot)$  by 
\begin{equation} \label{def[L_n]}
L_n(\bX_n,\Phi_n,f) 
:= \ \sum_{j=1}^n f(\bx^j)\varphi_j.
\end{equation}
% Let $B$ be a (quasi-)normed space of functions on $\TTd$. For $f \in B$,  we measure the recovery error by 
% $\|f - L_n(\bX_n,\Phi_n,f)\|_B$. Let $W \subset B$. 
%The proof of the upper bounds in \eqref{introduction[r_n,p=q=2]} and \eqref{introduction[r_n,p<q]} is 

Based on the B-spline quasi-interpolation representation \eqref{introduction[B-splineRepresentation]}, we construct linear sampling algorithms on  Smolyak grids induced by partial sums of the series in \eqref{introduction[B-splineRepresentation]} as follows.
For $m \in \NN$, the well known periodic Smolyak grid of points $G^d(m) \subset \TTd$ is defined as
\begin{equation}  \nonumber
G^d(m)
:= \
\{\by = 2^{-\bk} \bs: \bk \in \NNd, \ |\bk|_1 = m, \ \bs \in I(\bk)\}.
\end{equation}
Here and in what follows, we use the notations: 
$\ZZ_+:= \{s \in \ZZ: s \ge 0 \}$; $\bx \by := (x_1y_1,..., x_dy_d)$; 
$2^\bx := (2^{x_1},...,2^{x_d})$;
$|\bx|_1 := \sum_{i=1}^d |x_i|$ for $\bx, \by \in {\RR}^d$;
$[n]$ denotes the set of all natural numbers from $1$ to $n$; 
$x_i$ denotes the $i$th coordinate of $\bx \in \RR^d$, i.e., $\bx := (x_1,..., x_d)$. 
For $m \in {\ZZ}_+$, we define the operator $R_m$ by 
\begin{equation*}
R_m(f) 
:= \ 
\sum_{|\bk|_1 \le m} q_\bk(f)
\ = \
\sum_{|\bk|_1 \le m} \ \sum_{\bs \in I(\bk)} c_{\bk,\bs}(f)\, N_{\bk,\bs}.
\end{equation*} 
For functions $f$ on $\TTd$, $R_m$ defines the linear sampling algorithm 
on the Smolyak grid $G^d(m)$ 
\begin{equation*} 
R_m(f) 
\ = \ 
L_n(\bY_n,\Phi_n,f) 
\ = \ 
\sum_{\by \in G^d(m)} f(\by) \psi_{\by}, 
\end{equation*} 
where $n := \ |G^d(m)|$, $\bY_n:= \{\by \in G^d(m)\}$, $\Phi_n:= \{\varphi_{\by}\}_{\by \in G^d(m)}$  and for 
$\by =  2^{-\bk} \bs$,  $\varphi_\by$ are explicitly constructed as linear combinations of at most 
at most $m_0$ B-splines $N_{\bk,\bj}$ for some $m_0 \in \NN$ which is independent of $\bk,\bs,m$ and $f$.  

Let $1 < p,q < \infty$ and $\max(\frac{1}{p},\frac{1}{2}) < r < \ell$.
Then by using the B-spline quasi-interpolation representation \eqref{introduction[B-splineRepresentation]} with norm equivalence \eqref{introduction[normeqWrp]} we prove that
\begin{equation} \label{Intr[|f -  R_m(f)|_p<]}
\big\|f -  R_m(f) \big\|_q
\ \asymp \
\|f\|_{\Wrp} \times
\begin{cases}
2^{-rm} m^{(d-1)/2}, \ & p \ge q, \\
2^{-(r - 1/p + 1/q)m} \, \ & p < q,
\end{cases}
\quad  \forall f \in \Wrp. 
\end{equation}

Let us introduce the Smolyak sampling width $r^s_n(F)_q$ characterizing  optimality of sampling recovery  on Smolyak grids $G^d(m)$ with respect to the function class $F$ by
\begin{equation} \label{def:r^s_n}
r^s_n(F,L_q)
\ := \ \inf_{|G^d(m)| \le n, \, \Phi_m} \  \sup_{f \in F} \, \|f - S_m(\Phi_m,f)\|_q,
\end{equation}
where for a family $\Phi_m = \{\varphi_\by\}_{\by \in G^d(m)}$ of functions we define the linear sampling algorithm
$S_m(\Phi_m,\cdot)$ on Smolyak grids $G^d(m)$ by
\begin{equation*}
S_m(\Phi_m,f)
\ = \
\sum_{\by \in G^d(m)} f(\by) \varphi_\by.
\end{equation*}
The upper index $s$ indicates that we restrict to Smolyak grids here.
 Let $1 < p,q < \infty$ and $r > \max(\frac{1}{p},\frac{1}{2})$.
Denote by $\Urp$ the unit ball in the space $\Wrp$. Then we prove the asymptotic order
\begin{equation} \nonumber
r_n^s(\Urp,L_q) 
\ \asymp \
\begin{cases}
\biggl(\frac{(\log n)^{d-1}}{n}\biggl)^{r} (\log n)^{(d-1)/2}, \ & p \ge q, \\[1.5ex]
\biggl(\frac{(\log n)^{d-1}}{n}\biggl)^{(r-1/p+1/q)},  \, \ & p < q.
\end{cases}
\end{equation}
The upper bound follows from \eqref{Intr[|f -  R_m(f)|_p<]}, while the lower bound is established by construction of 
test functions which is based in the inverse theorem of sampling representation 
(see Theorem \ref{InverseThm} below). Moreover, the linear sampling algorithms $R_m(\cdot)$ on the Smolyak grid $G^d(m)$  for which 
$n := \ |G^d(m)|$, are asymptotically optimal for $r_n^s(\Urp,L_q) $.  

To study  optimality of linear sampling algorithms 
of the form \eqref{def[L_n]} for 
recovering $f \in F$ from $n$ of their values, we can use also the sampling width
\begin{equation} \nonumber 
r_n(F,L_q) 
\ := \ \inf_{\bX_n, \Phi_n} \  \sup_{f \in \Urp} \, \|f - L_n(\bX_n,\Phi_n,f)\|_q.
\end{equation}
For $1 < p < q \le 2$ or $2\le p < q < \infty$ and $r > \max(\frac{1}{p},\frac{1}{2})$, 
as a consequence of \eqref{Intr[|f -  R_m(f)|_p<]} and a result on linear width proven 
in  \cite{Ga87} we obtain the asymptotic order
\begin{equation}  \label{introduction[r_n,p<q]}
r_n(\Urp,L_q) 
\ \asymp \
\biggl(\frac{(\log n)^{d-1}}{n}\biggl)^{r-1/p + 1/q}.
\end{equation}

The sparse grids $G^d(m)$ for sampling recovery and numerical integration were first considered by Smolyak 
\cite{Sm63}. In \cite{Te85}--\cite{Te93b} and  
\cite{Di90}--\cite{Di92} Smolyak's construction was developed for studying the trigonometric sampling recovery and sampling width for periodic Sobolev classes  and Nikol'skii classes  having mixed smoothness. Recently, 
the sampling recovery  for Sobolev and Besov classes  having mixed smoothness has been investigated in  
\cite{ByDuUl14, ByUl15, Di11,Di13,DU14,SU07,SU11,U08}. 
In particular, for non-periodic functions of mixed smoothness linear sampling algorithms on Smolyak grids have been recently studied in  \cite{Tr10} $(d=2)$, \cite{Di11, Di13,DU14,SU11} using B-spline quasi-interpolation sampling representation. For $1 \le p,q \le \infty$, $0 < \theta \le \infty$ and $r > 1/p$, the asymptotic order of the Smolyak sampling width $r^s_n(U^r_{p,\theta},L_q)$ has been established in \cite{Di11,DU14} where $U^r_{p,\theta}$ is the Besov class of uniform mixed smoothness $r$.
The first asymptotic order of sampling width 
$r_n(U^r_{p,\infty},L_q)$ for 
$ 1 < p < q \le 2, \ r > 1/p$,
among classes of mixed smoothness was obtained in \cite{Di90,Di91}. 
For Sobolev classes of mixed smoothness, the first asymptotic order of sampling width
$r_n(U^r_2,L_\infty)$ was obtained in \cite{Te93}. 

It is remarkable to notice that so far the asymptotic orders of the sampling widths $r_n(U^r_{p,\theta},L_q)$ and $r_n(U^r_p,L_q)$ are known only in some cases with the condition $p < q$, for which the sampling algorithms $R_m$ on the Smolyak grid $G^d(m)$ are asymptotically optimal. Even the asymptotic order of the "simplest" sampling widths $r_n(U^r_2,L_2)$ is still an outstanding open problem.

In numerical applications for approximation problems involving a large number of variables, Smolyak grids was first considered in
\cite{Ze91}. For non-periodic functions of mixed smoothness of integer order, linear sampling algorithms on Smolyak grids have been investigated in \cite{BG04} employing hierarchical Lagrangian polynomials multilevel basis. There is a very large number of papers on Smolyak grids and their modifications in various problems of approximations, sampling recovery and integration with applications in data mining, mathematical finance, learning theory, numerical solving of PDE and stochastic PDE, etc. to mention all of them. The reader can see the surveys in \cite{BG04, GN09, GeGr08} and the references therein.

Quasi-interpolation based on scaled B-splines with integer knots and constructed from function values at dyadic lattices, possesses good local and approximation properties for smooth functions, see \cite[p. 63--65]{BHR93}, 
\cite[p. 100-107]{C92}. It can be an efficient tool in some high-dimensional approximation problems, especially in applications ones. Thus, one of the important bases for sparse grid high-dimensional approximations having various applications, are the Faber functions (hat functions) which are piecewise linear B-splines of second order  
\cite{BG04,GeGr08, GHo10, GH13a, GH13b, GaHe09, BGGK13}.  The representation by Faber basis can be obtained by the B-spline quasi-interpolation (see, e. g., \cite{Di11}). 

A central role in sampling recovery of functions having a mixed smoothness, or more generally an anisotropic smoothness play sampling representations which are based on dyadic scaled B-splines with integer knots or trigonometric kernels and constructed from function values at dyadic lattices. These representations are in the form of a B-spline or trigonometric polynomial series provided with discrete equivalent norm for functions in H\"older-Nikol'skii- and Besov-types spaces of a mixed smoothness. By employing them we can construct sampling algorithms for recovery on Smolyak-type grids of functions from the corresponding spaces which in some cases give the asymptotically optimal rate of the approximation error. While the quasi-interpolation  and trigonometric sampling representation theorems are already established for  H\"older-Nikol'skii- and Besov-types spaces of a mixed smoothness 
\cite{Di91,Di01,Di11,Di13}, they are almost not formulated for Sobolev-type spaces of a mixed smoothness. Only a few particular cases are known for a small uniform mixed smoothness \cite{ByUl16,Tr10}.  The relations 
\eqref{introduction[B-splineRepresentation]} and \eqref{introduction[normeqWrp]}  present a Littlewood-Paley-type theorem on B-spline quasi-interpolation sampling representation for periodic Sobolev-type spaces of arbitrary uniform mixed smoothness.  Independently from the present paper, a trigonometric counterpart of this theorem as well sampling recovery based on it have been investigated in \cite{ByUl15}. 

Finally, we refer the reader to the survey \cite{DTU16} for various aspects, recent development and bibliography on sampling recovery of functions having mixed smoothness and related problems.
 
The paper is organized as follows. In Section  \ref{B-spline quasi-interpolation}, we present a B-spline quasi-interpolation sampling representation for continuous functions on $\TTd$, and prove an explicit formula for the coefficient functionals in this representation. In Section \ref{Littlewood-Paley-type theorems}, we prove direct and inverse Littlewood-Paley-type theorem for Sobolev spaces $\Wrp$.  In Section \ref{Sampling recovery}, for 
Sobolev classes $\Urp$, we investigate the sampling recovery in $L_q$-norm by linear sampling algorithms induced by partial sums of  B-spline quasi-interpolation sampling representation, optimality of sampling recovery on Smolyak grids and the asymptotic order of $r^s_n(\Urp,L_q)$ and $r_n(\Urp,L_q)$.

\section{B-spline quasi-interpolation sampling representations} 
 \label{B-spline quasi-interpolation}

\subsection{B-spline quasi-interpolations and sampling representations} 
\label{B-spline Q-I}

In order to construct B-spline quasi-interpolation sampling representations for continuous functions on $\TTd$, we preliminarily introduce quasi-interpolation operators for functions on $\RRd$. For a given natural number $\ell,$ denote by  $M_\ell$ the cardinal B-spline of order $\ell$ with support $[0,\ell]$ and 
knots at the points $0, 1,...,\ell$. 
We fixed an even number $\ell \in \NN$ and take the cardinal B-spline $M= M_{\ell}$ of order $\ell$.
Let $\Lambda = \{\lambda(s)\}_{|j| \le \mu}$ be a given finite even sequence, i.e., 
$\lambda(-j) = \lambda(j)$ for some $\mu \ge \frac{\ell}{2} - 1$. 
We define the linear operator $Q$ for functions $f$ on $\RR$ by  
\begin{equation} \label{def:Q}
Q(f,x):= \ \sum_{s \in \ZZ} \Lambda (f,s)M(x-s), 
\end{equation} 
where
\begin{equation} \label{def:Lambda}
\Lambda (f,s):= \ \sum_{|j| \le \mu} \lambda (j) f(s-j + \ell/2).
\end{equation}
The operator $Q$ is local and bounded in $C(\RR)$  (see \cite[p. 100--109]{C92}).
An operator $Q$ of the form \eqref{def:Q}--\eqref{def:Lambda} is called a 
{\it quasi-interpolation operator in} $C(\RR)$ if  it reproduces 
$\Pp_{\ell-1}$, i.e., $Q(f) = f$ for every $f \in \Pp_{\ell-1}$, where $\Pp_{\ell-1}$ denotes the set of $d$-variate polynomials of degree at most $\ell-1$ in each variable.

If $Q$ is a quasi-interpolation operator of the form 
\eqref{def:Q}--\eqref{def:Lambda}, for $h > 0$ and a function $f$ on $\RR$, 
we define the operator $Q(\cdot;h)$ by
$
Q(f;h) 
:= \ 
\sigma_h \circ Q \circ \sigma_{1/h}(f),
$
where $\sigma_h(f,x) = \ f(x/h)$.
Let $Q$ be a quasi-interpolation operator of the form \eqref{def:Q}--\eqref{def:Lambda} in $C({\RR}).$ 
If $k \in \ZZ_+ $, we introduce the operator $Q_k$  by  
\begin{equation*}
Q_k(f,x) := \ Q(f,x;h^{(k)}), \  x \in \RR, \quad h^{(k)}:= \ \ell^{-1}2^{-k}. 
\end{equation*}
We define the integer translated dilation $M_{k,s}$ of $M$ by   
\begin{equation*}
M_{k,s}(x):= \ M(\ell2^k x - s), \ k \in {\ZZ}_+, \ s \in \ZZ. 
\end{equation*}
Then we have for $k \in {\ZZ}_+$,
\begin{equation} \nonumber
Q_k(f)(x)  \ = \ 
\sum_{s \in \ZZ} a_{k,s}(f)M_{k,s}(x), \ \forall x \in \RR, 
\end{equation}
where the coefficient functional $a_{k,s}$ is defined by
\begin{equation} \label{def[a_{k,s}(f)]}
a_{k,s}(f):= \ \Lambda(f,s;h^{(k)}) 
= \   
\sum_{|j| \le \mu} \lambda (j) f(h^{(k)}(s-j+r)).
\end{equation}
Notice that $Q_k(f)$ can be written in the form:
\begin{equation} \label{[L_{k,s}-representation]}
Q_k(f)(x)  \ = \ 
\sum_{s \in \ZZ} f(h^{(k)}(s+r))L_k(x-s), \ \forall x \in \RR, 
\end{equation}
where the function $L_k$ is defined by
\begin{equation} \label{L_{k,s}}
L_k:= \  
= \   
\sum_{|j| \le \mu} \lambda (j) M_{k,j}.
\end{equation}
From \eqref{[L_{k,s}-representation]} and \eqref{L_{k,s}} we get for a function $f$ on $\RR$,
\begin{equation} \label{[|Q_bk(f)|<]}
\|Q_k(f)\|_{C(\RR)} 
\ \le \ 
\|L_\Lambda\|_{C(\RR)} \|f\|_{C(\RR)}
\ \le \ 
\|\Lambda \|\|f\|_{C(\RR)},    
\end{equation}
where
\begin{equation} \label{[L]}
L_\Lambda(x):= \  
= \   
\sum_{s \in \ZZ} \sum_{|j| \le \mu} \lambda (j) M(x-j-s),
\quad
\|\Lambda \|= \ \sum_{|j| \le \mu} |\lambda (j)|.
\end{equation} 
 
For $\bk \in \ZZdp$, let the mixed operator $Q_\bk$ be defined by
\begin{equation} \label{def:Mixed[Q_bk]} 
Q_\bk:= \prod_{i=1}^d  Q_{k_i},
\end{equation}
where the univariate operator
$Q_{k_i}$ is applied to the univariate function $f$ by considering $f$ as a 
function of  variable $x_i$ with the other variables held fixed.
We define the $d$-variable B-spline $M_{k,s}$ by
\begin{equation} \label{def:Mixed[M_{k,s}]}
M_{\bk,\bs}(\bx):=  \ \prod_{i=1}^d M_{k_i, s_i}( x_i),  
\ \bk \in {\ZZ}^d_+, \ \bs \in {\ZZ}^d,
\end{equation}
where ${\ZZ}^d_+:= \{\bs \in {\ZZ}^d: s_i \ge 0, \ i \in [d] \}$.
Then we have
\begin{equation*} 
Q_\bk(f,\bx)  \ = \ 
\sum_{s \in \ZZ^d} a_{\bk,\bs}(f)M_{\bk,\bs}(\bx), \quad \forall \bx \in \RRd, 
\end{equation*}
where $M_{\bk,\bs}$ is the mixed B-spline  defined in \eqref{def:Mixed[M_{k,s}]}, 
and
\begin{equation} \label{def:Mixed[a_{k,s}(f)]}
a_{\bk,\bs}(f) 
\ = \   
\biggl(\prod_{j=1}^d a_{k_j,s_j}\biggl)(f),
\end{equation}
and the univariate coefficient functional
$a_{k_i,s_i}$ is applied to the univariate function $f$ by considering $f$ as a 
function of  variable $x_i$ with the other variables held fixed. 

Since $M(\ell\,2^k x)=0$ for every $k \in \ZZ_+$ and $x \notin (0,1)$, we can extend  the univariate B-spline  
$M(\ell\,2^k\cdot)$ to an $1$-periodic function on the whole $\RR$. Denote this periodic extension by $N_k$ and define 
\begin{equation*}
N_{k,s}(x):= \ N_k(x - h^{(k)}s), \ k \in {\ZZ}_+, \ s \in I(k), 
\end{equation*}
where $I(k) := \{0,1,..., \ell2^k - 1\}$.
We define the $d$-variable B-spline $N_{\bk,\bs}$ by
\begin{equation} \nonumber
N_{\bk,\bs}(\bx):=  \ \prod_{i=1}^d N_{k_i, s_i}( x_i),  \ \bk \in {\ZZ}^d_+, \ \bs \in I(\bk),
\end{equation}
where $ I(\bk):=\prod_{i=1}^d I(k_i)$.
Then we have for functions $f$ on $\TTd$,
\begin{equation} \label{def[periodicQI]}
Q_\bk(f,\bx)  \ = \ 
\sum_{\bs \in I(\bk)} a_{\bk,\bs}(f) N_{\bk,\bs}(x), \quad \forall x \in \TTd. 
\end{equation}
Since the function $L_\Lambda$ defined in \eqref{[L]} is $1$-periodic, from \eqref{[|Q_bk(f)|<]} it follows that for a function $f$ on $\TT$,
\begin{equation} \label{[|Q_k(f)|<(T)]}
\|Q_\bk(f)\|_{C(\TT)} 
\ \le \ 
\|L_\Lambda\|_{C(\TT)} \|f\|_{C(\TT)}
\ \le \ 
\|\Lambda \|\|f\|_{C(\TT)},    
\end{equation}

For $\bk \in \ZZdp$, we write  $\bk  \to  \infty$ if $k_i  \to  \infty$ for $i \in [d]$). In a way similar to the proof of \cite[Lemma 2.2]{Di13} one can show that  for every $f \in C(\TTd)$,
\begin{equation} \label{ConvergenceMixedQ_bk(f)}
\|f - Q_\bk(f)\|_{C(\TTd)} \to 0 , \ \bk  \to  \infty.
\end{equation}

For convenience we define the univariate operator $Q_{-1}$ by putting $Q_{-1}(f)=0$ for all $f$ on $\II$. Let  the operators $q_\bk$ be defined 
in the manner of the definition \eqref{def:Mixed[Q_bk]} by
\begin{equation} \label{eq:Def[Q_bk]}
q_\bk := \ \prod_{i=1}^d \biggl(Q_{k_i}- Q_{k_i-1}\biggl), \ \bk \in {\ZZ}^d_+. 
\end{equation}

From the equation $Q_\bk = \sum_{\bk' \le \bk}q_{\bk'}$
and \eqref{ConvergenceMixedQ_bk(f)} it is easy to see that 
a continuous function  $f$ has the decomposition
$
f \ =  \ \sum_{\bk \in {\ZZ}^d_+} q_\bk(f)
$
with the convergence in the norm of $C(\TTd)$.  
From  the refinement equation for the B-spline $M$, in the univariate case, we can represent the component functions $q_\bk(f)$ as 
 \begin{equation} \label{eq:RepresentationMixedQ_bk(f)}
q_\bk(f) 
= \ \sum_{\bs \in I(\bk)}c_{\bk,\bs}(f) N_{\bk,\bs},
\end{equation}
where $c_{\bk,\bs}$ are certain coefficient functionals of 
$f$. In the multivariate case, the representation  \eqref{eq:RepresentationMixedQ_bk(f)} holds true 
with the $c_{k,s}$ which are defined in the manner of the definition  
\eqref{def:Mixed[a_{k,s}(f)]} by
\begin{equation} \nonumber
c_{\bk,\bs}(f) 
\ = \   
\biggl(\prod_{j=1}^d c_{k_j,s_j}\biggl)(f).
\end{equation}
See \cite{Di11} for details. Thus, we have proven the following periodic B-spline quasi-interpolation representation for continuous functions on $\TTd$.
\begin{lemma} \label{lemma[representation]}
Every continuous function $f$ on $\TTd$ is represented as B-spline series 
\begin{equation} \label{eq:B-splineRepresentation}
f \ = \sum_{\bk \in {\ZZ}^d_+} \ q_\bk(f) = 
\sum_{\bk \in {\ZZ}^d_+} \sum_{\bs \in I(\bk)} c_{\bk,\bs}(f)N_{\bk,\bs}, 
\end{equation}  
converging in the norm of $C(\TTd)$, where the coefficient functionals $c_{\bk,\bs}(f)$ are explicitly constructed as
linear combinations of at most $m_0$ of function
values of $f$ for some $m_0 \in \NN$ which is independent of $\bk,\bs$ and $f$.
\end{lemma}

% The formula \eqref{def[periodicQI]} with the coefficients valued functionals $a_{k,s}(f)$ given as in 
% \eqref{def[a_{k,s}(f)]} and \eqref{def:Mixed[a_{k,s}(f)]}, defines a {\em multivariate periodic quasi-interpolation operator} for functions on $\TTd$. While the formula \eqref{eq:B-splineRepresentation} with the coefficients valued functionals $a_{k,s}(f)$ given as in \eqref{eq:Def[Q_\bk]}, defines a {\em multivariate periodic B-spline quasi-interpolation representation} for functions on $\TTd$. They are completely generated from the initial quasi-interpolation operator $Q$ given as in \eqref{def:Q} and \eqref{def:Lambda}. 
\subsection{A formula for the coefficients in B-spline quasi-interpolation representations} 
\label{Coefficient}

In this subsection, we find an explicit formula for the coefficients 
$c_{\bk,\bs}(f)$ in the representation \eqref{eq:B-splineRepresentation} related to the $\ell$th difference operator $\Delta_\bh^\ell$, which plays an important role in the proof of a direct theorem of sampling representation of functions in the Sobolev space $\Wrp$. 

For univariate functions $f$ on $\TT$ the $\ell$th difference operator $\Delta_h^\ell$ is defined by 
\begin{equation} \label{def[Delta_h^ell]}
\Delta_h^\ell(f,x) := \
\sum_{j =0}^\ell (-1)^{\ell - j} \binom{\ell}{j} f(x + jh).
\end{equation}
If $u$ is any subset of $[d]$, for multivariate functions on $\TTd$
the mixed $(\ell,u)$th difference operator $\Delta_\bh^{\ell,u}$ is defined by 
\begin{equation*}
\Delta_\bh^{\ell,u} := \
\prod_{i \in u} \Delta_{h_i}^\ell, \quad \Delta_\bh^{\ell,\varnothing} := \ I,
\end{equation*}
where the univariate operator
$\Delta_{h_i}^\ell$ is applied to the univariate function $f$ by considering $f$ as a 
function of  variable $x_i$ with the other variables held fixed, and $I(f):= f$ for functions $f$ on $\TTd$. 
We also use the abbreviation $\Delta_\bh^\ell:= \Delta_\bh^{\ell,[d]}$.

If $\bh \in \RRd$, we define the shift operator $T_\bh$ for functions $f$ on $\TTd$ by
$
T_\bh(f) := \
f(\cdot + \bh).
$
Recall that a $d$-variate Laurent polynomial is call a function $P$ of the form
\begin{equation} \label{def[P]}
P(\bz) = \
\sum_{\bs \in A}  c_\bs \bz^\bs,
\end{equation}
where $A$ is a finite subset in $\ZZd$ and $\bz^\bs:= \prod_{j=1}^d z_j^{s_j}$.
A $d$-variate Laurent polynomial $P$ as \eqref{def[P]} generates the operator $T_\bh^{[P]}$ by 
\begin{equation} \label{def[T_h^{[P]}]}
T_\bh^{[P]}(f) = \
\sum_{\bs \in A}  c_\bs T_{\bs\bh}(f).
\end{equation}
Sometimes we also write $T_\bh^{[P]}=T_\bh^{[P(\bz)]}$.
Notice that any operation over polynomials generates a corresponding operation over operators $T_\bh^{[P]}$. Thus, in particular, we have 
\begin{equation*}
T_\bh^{[a_1P_1 + a_2P_2]}(f) \ = \
a_1 T_\bh^{[P_1]}(f) + a_2 T_\bh^{[P_2]}(f), \quad T_\bh^{[P_1.P_2]}(f) \ = \ T_\bh^{[P_1]}\circ T_\bh^{[P_2]}(f).
\end{equation*}
By definitions we have 
\begin{equation*}
\Delta_\bh^\ell
= \
T_\bh^{[D_\ell]}, \  D_\ell:= \prod_{j=1}^d (z_j - 1)^\ell, \quad
\Delta_\bh^{\ell,u}
= \
T_\bh^{[D_{\ell,u}]}, \  D_{\ell,u}:= \prod_{j \in u}(z_j - 1)^\ell.
\end{equation*}
We say that a $d$-variate polynomial is a {\em tensor product polynomial} if it is of the form
$
P(\bz) = \
\prod_{j=1}^d P_j(z_j),
$
where $P_j(z_j)$ are univariate polynomial in variable $z_j$. 

\begin{lemma} \label{lemma[factor]}
Let $P$ be a tensor product Laurent polynomial, $\bh \in \RRd$ with $h_j \not=0$, and $\ell \in \NN$. Assume that $T_\bh^{[P]}(g)=0$ for every polynomial $g \in \Pp_{\ell-1}$ Then $P$ has a factor $D_\ell$ and consequently,
\begin{equation*}
T_\bh^{[P]}
\ = \
T_\bh^{[P^*]} \circ \Delta_\bh^\ell, \quad  P(\bz)= D_\ell(\bz)P^*(\bz),
\end{equation*}
where $P^*$ is a tensor product Laurent polynomial.
\end{lemma}

\begin{proof} By the tensor product argument it is enough to prove the lemma for the case $d=1$. We prove this case by induction on $l$. Let $P(z)=\sum_{s = -m}^n  c_s z^s$ for some $m,n \in \ZZ_+$. Consider first the case $l=1$. Assume that $T_h^{[P]}(g)=0$ for every constant functions $g$. Then 
replacing by $g_0=1$ in \eqref{def[T_h^{[P]}]} we get
$
T_h^{[P]}(g_0)
\ = \ 
\sum_{s = -m}^n  c_s \ = \ 0.
$
By B\'ezout's theorem $P$ has a factor $(z-1)$. This proves the lemma for $l=1$. Assume it is true for $l-1$ and $T_h^{[P]}(g)=0$ for every polynomial $g$ of degree at most $l-1$. By the induction assumption we have 
\begin{equation} \label{[induction]}
T_h^{[P]}
\ = \
T_h^{[P_1]} \circ \Delta_h^{l-1}, \quad  P(z):= (z-1)^{l-1} P_1(z).
\end{equation}
We take a proper polynomial $g_l$ of degree $l-1$ (with the nonzero eldest coefficient). Hence 
$\psi_l = \Delta_h^{l-1}(g_l) = a$ where $a$ is a nonzero constant. Similarly to the case $l=1$, from the equations
$ 0 \ = \ T_h^{[P]} (g_l) \ = \ T_h^{[P_1]} (\psi_l)$ we conclude that $P_1$ has a factor $(z-1)$. Hence, by 
\eqref{[induction]} we can see that $P$ has a factor $(z-1)^l$. The lemma is proved.
\end{proof}

Let us return to the definition of quasi-interpolation operator $Q$ of the form 
\eqref{def:Q} induced by the sequence $\Lambda$  as in \eqref{def:Lambda} which can be uniquely characterized by the univariate symmetric Laurent polynomial
\begin{equation} \label{P_Lambda}
P_\Lambda(z)
:= \ 
z^{\ell/2} \sum_{|s|\le \mu}\lambda (s) z^s.
\end{equation}
Let the $d$-variate symmetric tensor product Laurent polynomial $P_\Lambda$ be given by 
\begin{equation} \nonumber
P_\Lambda(\bz)
:= \ 
 \prod_{j=1}^d z_j^{\ell/2}\sum_{|s_j|\le \mu}\lambda (s_j) z_j^{s_j}.
\end{equation}
For the periodic quasi-interpolation operator 
\[
q_\bk(f)  \ = \ 
\sum_{\bs \in I(\bk)} a_{\bk,\bs}(f) N_{\bk,\bs} 
\]
given as in \eqref{def[periodicQI]}, from \eqref{def[a_{k,s}(f)]} we get
\begin{equation} \label{eq[a_{k,s}(f)]}
a_{\bk,\bs}(f)
\ = \ 
T_{\bh^{(\bk)}}^{[P_\Lambda]}(f)\big(\bs\bh^{(\bk)}\big),
\end{equation}
where $\bh^{(\bk)}:= \big(h_1^{(k_1)},...,h_d^{(k_d)}\big)$.

Let us first find an explicit formula for the univariate operator $Q_k(f)$.
We have for $k > 0$,
\begin{equation*}
\begin{aligned}
Q_k(f)
\ & = \
\sum_{s \in I(k)} T_{h^{(k)}}^{[P_\Lambda]}(f)(sh^{(k)})) N_{k,s} \\[1.5ex]
\ & = \
\sum_{s \in I(k-1)} T_{h^{(k)}}^{[P_\Lambda]}(f)(2sh^{(k)}))  N_{k,2s} 
\  +  \
\sum_{s \in I(k-1)} T_{h^{(k)}}^{[P_\Lambda]}(f)((2s+1)h^{(k)})) 
 N_{k,2s+1}. 
\end{aligned}
\end{equation*}
From \eqref{eq[a_{k,s}(f)]} and the refinement equation for $M$, we deduce that 
\begin{equation*}
\begin{aligned}
Q_{k-1}(f)
\ & = \
\sum_{s \in I(k-1)} T_{h^{(k-1)}}^{[P_\Lambda]}(f)(sh^{(k-1)})) 
\biggl[2^{-\ell+1}\sum_{j =0}^{\ell}\binom{\ell}{j} N_{k,2s+j}\biggl]\\[1.5ex]
\ & = \
2^{-\ell+1}\sum_{j =0}^r\binom{\ell}{2j}\sum_{s \in I(k-1)} T_{h^{(k-1)}}^{[P_\Lambda]}(f)(sh^{(k-1)})) 
 N_{k,2s+2j} \\[1.5ex]
\ & +  
2^{-\ell+1}\sum_{j =0}^{r-1}\binom{\ell}{2j+1}\sum_{s \in I(k-1)} T_{h^{(k-1)}}^{[P_\Lambda]}(f)(sh^{(k-1)})) 
 N_{k,2s+2j+1} \\[1.5ex]
\ & =: \
Q_{k-1}^{\operatorname{even}}(f) + Q_{k-1}^{\operatorname{odd}}(f).
\end{aligned}
\end{equation*}
By the identities $h^{(k-1)}= 2 h^{(k)}$, $N_{k,\ell2^k +m} = N_{k,m}$ and $f(h^{(k)})(\ell2^k + m)= f(h^{(k)}m)$ for 
$k \in \ZZ_+$ and $m \in \ZZ$, we have
\begin{equation*}
\begin{aligned}
Q_{k-1}^{\operatorname{even}}(f)
\ & = \
2^{-\ell+1}\sum_{j =0}^r\binom{\ell}{2j}\sum_{s \in j+I(k-1)} T_{h^{(k)}}^{[P_\Lambda]}(f)(2(s-j)h^{(k)})) 
 N_{k,2s} \\[1.5ex]
\ &= \  
2^{-\ell+1}\sum_{j =0}^r\binom{\ell}{2j}\sum_{s \in I(k-1)} T_{h^{(k)}}^{[P_\Lambda]}(f)(2(s-j)h^{(k)})) 
 N_{k,2s} \\[1.5ex]
\ &= \  
\sum_{s \in I(k-1)} T_{h^{(k)}}^{[P_{\operatorname{even}}']}(f)(2sh^{(k)}) N_{k,2s}, \\
\end{aligned}
\end{equation*}
where
\begin{equation} \label{P'_even}
\ P_{\operatorname{even}}'(z) 
:= \
2^{-\ell+1}P_\Lambda (z^2)\sum_{j =0}^r\binom{\ell}{2j}z^{-2j}
\end{equation}
In a similar way we obtain
\begin{equation*}
Q_{k-1}^{\operatorname{odd}}(f)
\  = \
\sum_{s \in I(k-1)} T_{h^{(k)}}^{[P_{\operatorname{odd}}']}(f)((2s+1)h^{(k)}) N_{k,2s+1}, 
\end{equation*}
where
\begin{equation} \label{P'_odd}
\ P_{\operatorname{odd}}'(z) 
:= \
2^{-\ell+1}P_\Lambda (z^2)\sum_{j =0}^{r-1}\binom{\ell}{2j+1}z^{-2j-1}.
\end{equation}
We define
\begin{equation} \label{P_even,P_odd}
\ P_{\operatorname{even}}  
:= \
P_\Lambda - P_{\operatorname{even}}', \quad
\ P_{\operatorname{odd}}  
:= \
P_\Lambda - P_{\operatorname{odd}}'
\end{equation}
Then from the definition $q_k(f) = Q_k(f) - Q_{k-1}(f)$ we receive the following representation for $q_k(f)$,
\begin{equation} \label{eq[q_0(f)]}
q_0(f) 
 \ = \ 
\sum_{s \in I(0)} T_{h^{(0)}}^{[P_\Lambda]}(f)(sh^{(0)})) N_{0,s}, 
 \end{equation}
 and for $k>0$,
\begin{equation} \label{eq[Q_bk(f)]}
q_k(f) 
 \ = \ 
q_k^{\operatorname{even}}(f) + q_k^{\operatorname{odd}}(f)
 \end{equation}
 with
\begin{equation} \nonumber
\begin{aligned}
q_k^{\operatorname{even}}(f)
&= \
\sum_{s \in I(k-1)} T_{h^{(k)}}^{[P_{\operatorname{even}}]}(f)(2sh^{(k)}) N_{k,2s},\\[1.5ex]
q_k^{\operatorname{odd}}(f)
&= \ 
\sum_{s \in I(k-1)} T_{h^{(k)}}^{[P_{\operatorname{odd}}]}(f)((2s+1)h^{(k)}) N_{k,2s+1}.
\end{aligned}
\end{equation}

From the definitions of $Q_k$ and $q_k$ it follows that 
\[T_{h^{(k)}}^{[P_{\operatorname{even}}]}(g)(2sh^{(k)}) = 0 \quad \text{and} \quad  
T_{h^{(k)}}^{[P_{\operatorname{odd}}]}(g)((2s+1)h^{(k)}) = 0 \quad \text{for every} \quad g \in \Pp_r.
\] 
Hence, by 
Lemma \ref{lemma[factor]} we prove the following lemma for the univariate operators $q_k$.

\begin{lemma} \label{lemma[Q_bk^even&odd]}
We have
\begin{equation} \label{eq[Q_bk^even&odd]}
\begin{aligned}
P_{\operatorname{even}}(z)
&= \
D_{\ell}(z) P_{\operatorname{even}}^*(z)\,\\[1.5ex]
P_{\operatorname{odd}}(z)
&= \ 
D_{\ell}(z) P_{\operatorname{odd}}^*(z). 
\end{aligned}
\end{equation}
where $P_{\operatorname{even}}^*$, $P_{\operatorname{odd}}^*$ are a symmetric Laurent polynomial.
Therefore, in the representation \eqref{eq[q_0(f)]}--\eqref{eq[Q_bk(f)]} of $Q_k(f)$, we have  for $k>0$,
\begin{equation} \nonumber
\begin{aligned}
q_k^{\operatorname{even}}(f)
&= \
\sum_{s \in I(k-1)} T_{h^{(k)}}^{[P_{\operatorname{even}}^*]}\circ 
\Delta_{h^{(k)}}^{\ell} (f)(2sh^{(k)}) N_{k,2s},\\[1.5ex]
q_k^{\operatorname{odd}}(f)
&= \ 
\sum_{s \in I(k-1)} T_{h^{(k)}}^{[P_{\operatorname{odd}}^*]}\circ 
\Delta_{h^{(k)}}^{\ell} (f)((2s+1)h^{(k)}) N_{k,2s+1}.
\end{aligned}
\end{equation}
Equivalently, in the representation \eqref{eq:RepresentationMixedQ_bk(f)} of $Q_k(f)$, we have for $s \in I(0)$
\begin{equation} \nonumber
c_{0,s}(f) 
 \ = \ 
T_{h^{(0)}}^{[P_\Lambda]}(f)(sh^{(0)}),
 \end{equation}
 and for $k>0$ and $s \in I(k)$,
 \begin{equation} \nonumber
c_{k,s}(f) 
 \ = \
\begin{cases}
T_{h^{(k)}}^{[P_{\operatorname{even}}^*]}\circ \Delta_{h^{(k)}}^{\ell} (f)(sh^{(k)}),  & s \ \text{even} \\[1.5ex]
T_{h^{(k)}}^{[P_{\operatorname{odd}}^*]}\circ \Delta_{h^{(k)}}^{\ell} (f)(sh^{(k)}), & s \ \text{odd}.
\end{cases} 
\end{equation}
\end{lemma}

\begin{proof} Consider the representation \eqref{eq:RepresentationMixedQ_bk(f)}  for $Q_k(f)$ and $d=1$. If $g$ is arbitrary polynomial of degree at most $\ell-1$, then since $Q_k$ reproduces $g$ we have $Q_k(g) = 0$ and consequently, 
$c_{k,s}(g)=0$ for $k > 0$. The equations \eqref{eq[q_0(f)]}--\eqref{eq[Q_bk(f)]} give an explicit formula for the coefficient $c_{k,s}(g)$ as $T_{h^{(k)}}^{[P_{\operatorname{even}}]}(g)(2sh^{(k)})$ and $T_{h^{(k)}}^{[P_{\operatorname{odd}}]}(g)((2s+1)h^{(k)})$. Hence, by Lemma \ref{lemma[factor]} we get 
\eqref{eq[Q_bk^even&odd]}.
\end{proof}

 Put  $\ZZ_+:= \{s \in \ZZ: s \ge 0 \}$ and  
$\ZZdp(u):= \{\bs \in \ZZdp: s_i = 0 , \ i \notin u\}$  for a set $u \subset [d]$. 
\begin{theorem} \label{theorem[c_{k,s}]}
In the representation \eqref{eq:RepresentationMixedQ_bk(f)} of $Q_\bk(f)$, we have for every 
$\bk \in \ZZdp(u)$ and $\bs \in I(\bk)$,
\begin{equation} \label{eq[c_{k,s}(f)](d>1)}
c_{\bk,\bs}(f) 
 \ = \ 
T_{\bh^{(\bk)}}^{[P_{\bk,\bs}]}(f)(\bs\bh^{(\bk)}), 
\end{equation}
where
\begin{equation} \label{eq[P_{k,s}]}
P_{\bk,\bs} (\bz)
 \ = \
\prod_{j \not\in u} P_\Lambda(z_j)\,\prod_{j \in u} P_{k_j,s_j}^*(z_j) \prod_{j \in u} D_{\ell}(z_j), 
\end{equation}
\begin{equation} \label{eq[P_{k,s}^*]}
P_{k_j,s_j}^*(z_j) 
 \ = \
\begin{cases}
P_{\operatorname{even}}^*(z_j),  & s \ \text{even}, \\[1.5ex]
P_{\operatorname{odd}}^*(z_j), & s \ \text{odd}.
\end{cases} 
\end{equation}
\end{theorem}

\begin{proof} Indeed, from the definition of $c_{\bk,\bs}(f)$ and Lemma \ref{lemma[Q_bk^even&odd]} we have
for every $\bk \in \ZZdp(u)$ and $\bs \in I(\bk)$,
\begin{equation} \nonumber
c_{\bk,\bs}(f) 
 \ = \ 
\biggl(\prod_{j=1}^d T_{h^{(k)}_j}^{[P_{k_j,s_j}]}\biggl)(f)(\bs\bh^{(\bk)}), 
 \ = \ 
T_{\bh^{(\bk)}}^{[P_{\bk,\bs}]}(f)(\bs\bh^{(\bk)}),
\end{equation}
where
\begin{equation} \nonumber
P_{k_j,s_j}(z_j) 
 \ = \
\begin{cases}
P_\Lambda(z_j), \ & k_j=0, \\[1.5ex] 
P_{\operatorname{even}}^*(z_j) D_{\ell}(z_j), \ & k_j > 0, \ s \ \text{even} \\[1.5ex]
P_{\operatorname{odd}}^*(z_j) D_{\ell}(z_j), \ & k_j > 0,\ s \ \text{odd}.
\end{cases} 
\end{equation}
\end{proof}

\subsection{Examples}
\label{Examples}

The operator $Q$ is induced by the sequence $\Lambda$  as in \eqref{def:Lambda} which can be uniquely characterized by the univariate symmetric Laurent polynomial
$P_\Lambda$. In this subsection, we give some examples of the univariate symmetric Laurent polynomial $P_\Lambda$ characterizing quasi-interpolation operator $Q$ of the form  \eqref{def:Q}. For a given $P_\Lambda$, the Laurent polynomials $P_{\operatorname{even}}^*$ and  $P_{\operatorname{odd}}^*$ can be computed from 
\eqref{P'_even}--\eqref{P_even,P_odd}.

Let us consider the case $\ell = 2$ when $M(x)\ = \ (1 - |x-1|)_+$ is the piece-wise linear cardinal B-spine with knot at $0,1,2$. Let $\Lambda = \{\lambda(s)\}_{j=0}$ $(\mu=0)$ be a given by $\lambda(0)= 1$. 
If $N_k$ is the periodic extension of $M(2^{k+1}\cdot)$,  then 
\begin{equation*}
N_{k,s}(x):= \ N_k(x - s), \ k \in {\ZZ}_+, \ s \in I(k), 
\end{equation*}
where $I(k) := \{0,1,..., 2^{k+1} - 1\}$.
Consider the related periodic nodal quasi-interpolation operator for functions $f$ on $\TT$ and $k \in \ZZ_+$,
\begin{equation} \nonumber
Q_k(f,x)= \ \sum_{s \in I(k)} f(2^{-(k+1)}(s+1)) N_{k,s}(x) 
\end{equation}
We have
\begin{equation} \nonumber
\begin{aligned}
P_\Lambda(z)
\ &= \ 
z, \\[1.5ex]
P_{\operatorname{even}}(z)
\ &= \
- \frac{1}{2} (z-1)^2, \quad 
P_{\operatorname{even}}^*(z)
\ = \
- \frac{1}{2},\\[1.5ex]
P_{\operatorname{odd}}(z)
&= \ 
P_{\operatorname{odd}}^*(z)
\ = \ 
0,
\end{aligned}
\end{equation}
and
\begin{equation} \nonumber
\|L_\Lambda\|
\ = \ 
1, \quad
\|P_{\operatorname{even}}^*\|
\ = \ 
\frac{1}{2}, \quad
\|P_{\operatorname{odd}}^*\|
\ = \ 
0.
\end{equation}
Hence,
\begin{equation} \nonumber
q_0(f) 
 \ = \ 
\sum_{s=0}^1 T_{2^{-1}}^{[P_\Lambda]}(f)(2^{-1}s) N_{0,s}
\ = \
f(0), 
 \end{equation}
 and for $k>0$,
\begin{equation} \nonumber
q_k(f) 
 \ = \ 
q_k^{\operatorname{even}}(f) 
\ = \
\sum_{s =0}^{2^k-1} 
\biggl\{- \frac {1}{2} \Delta_{2^{-(k+1)}}^2 f(2^{-k}s)\biggl\} N_{k,2s}.
\end{equation}
We show that after redefining $N_{k,2s}$ as $\varphi_{k,s}$, the quasi-interpolation representation 
\eqref{eq:B-splineRepresentation} becomes the classical periodic Faber series.
We introduce the univariate hat functions $\varphi_{k,s}$ by
\begin{equation*} 
\varphi_{0,0}:= 1, 
\quad
\varphi_{k,s}:=  N_{k,2s}, \ k > 0 \ s \in Z(k),
\end{equation*}
where $Z(0) := \{0\}$ and  $Z(k) := \{0,1,..., 2^{k-1} - 1\}$.
Put $Z(\bk):= \prod_{i=1}^d Z(k_i)$. 
For $\bk \in {\ZZ}^d_+$, $\bs \in Z(\bk)$, define the $d$-variate hat functions
\begin{equation*} 
\varphi_{\bk,\bs}(\bx)
\ := \
\prod_{i=1}^d \varphi_{k_i,s_i}(x_i),
\end{equation*}
and the $d$-variate periodic Faber system $\Ff_d$  by 
\begin{equation*} 
\Ff_d := \{\varphi_{\bk,\bs}: \bs \in Z(\bk),\ \bk \in \ZZdp\}. 
\end{equation*}
For functions $f$ on $\TT$, we define the univariate linear functionals $\lambda_{k,s}$ by 
\begin{equation*} 
\lambda_{k,s}(f) \ := \
- \frac {1}{2} \Delta_{2^{-k}}^2 (f,2^{-k + 1}s),   \, k > 0, \ 
\text{and} \ \lambda_{0,0}(f) \ := \ f(0). 
\end{equation*}
Let the $d$-variate linear functionals
$\lambda_{k,s}$ be defined as
\begin{equation*}   
\lambda_{\bk,\bs}(f) 
\ := \   
\lambda_{k_1,s_1}(\lambda_{k_2,s_2}(... \lambda_{k_d,s_d}(f))),
\end{equation*}
where the univariate functional
$\lambda_{k_i,s_i}$ is applied to the univariate function $f$ by considering $f$ as a 
function of  variable $x_i$ with the other variables held fixed. 
It is well known that the $d$-variate periodic Faber system $\Ff_d$ is a basis in $C(\TTd)$, and a function 
$f \in C(\TTd)$ can be represented by the Faber series 
\begin{equation} \label{eq[FaberRepresentation]}
f
\ = \
\sum_{\bk \in \ZZdp} q_\bk(f) 
\ = \
\sum_{\bk \in \ZZdp} \sum_{\bs \in I(\bk)} \lambda_{\bk,\bs}(f)\varphi_{\bk,\bs}, 
\end{equation}
converging in the norm of $C(\TTd)$. 

% Put $\ZZdpu:= \{\bk \in \ZZdp: \operatorname{supp} (\bk) = u\}$ for $u \subset [d]$, and use the convention $x^0 = 1$ for $x \in [0, \infty]$. 

Let us consider the case $\ell = 4$ when $M(x)$ is the cubic cardinal B-spine with knot at $0,1,2,3,4$. 
We define a sequence $\Lambda$ of the form  \eqref{def:Lambda}  inducing a quasi-interpolation $Q$ 
via the polynomial $P_\Lambda$ as in \eqref{P_Lambda} which uniquely defines $\Lambda$.
One of possible choices  is
\begin{equation} \nonumber
P_\Lambda(z)
\ = \ 
\frac{z^2}{6}(- z + 8 - z^{-1})
\ = \ 
- \frac{1}{6} z^3 + \frac{8}{6} z^2 - \frac{1}{6} z.
\end{equation}
Then, we have
\begin{equation} \nonumber
\begin{aligned}
P_{\operatorname{even}}(z)
\ &= \
(z-1)^4 P_{\operatorname{even}}^*(z), \quad 
P_{\operatorname{even}}^*(z)
\ = \
\frac{1}{48}z^{-2}\biggl( z^4 + 4z^3 + 8z^2 + 4z + 1 \biggl),\\[1ex]
P_{\operatorname{odd}}(z)
&:= \ 
(z-1)^4 P_{\operatorname{odd}}^*(z), \quad
P_{\operatorname{odd}}^*(z)
:= \ 
\frac{1}{12} \biggl( z^2 + 4z + 1 \biggl),
\end{aligned}
\end{equation}
and 
\begin{equation} \nonumber
\|L_\Lambda\|
\ = \ 
\frac{11}{9}, \quad
\|P_{\operatorname{even}}^*\|
\ = \ 
\frac{3}{8}, \quad
\|P_{\operatorname{odd}}^*\|
\ = \ 
\frac{1}{2}.
\end{equation}

\section{Direct and inverse theorems of sampling representation}
\label{Littlewood-Paley-type theorems}

\subsection{Function spaces of  mixed smoothness} 
\label{Function space}

We define the univariate Bernoulli kernel
\[ 
F_r(x)
:= \
1 + 2\sum_{k=1}^\infty k^{-r}\cos (kx - r \pi/2), \quad x \in \TT,
\]
 and the multivariate Bernoulli kernels as the corresponding tensor products 
\begin{equation} \nonumber
F_r(\bx)
:= \
\prod_{j=1}^d F_r(x_j),\quad \bx \in
\TTd.
\end{equation}
Let $r > 0$ and $0< p \le \infty$. 
Denote by  
$L_p = L_p(\TTd)$ the quasi-normed space 
of functions on $\TTd$ with the $p$th integral quasi-norm 
$\|\cdot\|_p$ for $0 < p < \infty,$ and 
the ess sup-norm $\|\cdot\|_p$ for $p = \infty$.
If $r > 0$ and $1 \le p \le \infty$, we define the Sobolev space $\Wrp$ of mixed smoothness $r$ by
\begin{equation} \label{def[Wrp]}
\Wrp
:= \
\Big\{f \in L_p:\, 
f \, = \, F_r\ast \varphi:= \, \int_{\TTd}F_r(\bx-\by)\varphi(\by)\operatorname{d}\by ,\quad
\|\varphi\|_p < \infty \Big\},
\end{equation}
and $\|f\|_{\Wrp}:= \|\varphi\|_p$ for $f$ represented as in \eqref{def[Wrp]}.
For $1 < p < \infty$, the  space $\Wrp$ coincides with the
set of all $f\in L_p$ such that the norm
\[
\|f\|_{\Wrp} 
:= \
\Big\| \sum_{\bs \in \ZZd} \, \hat{f}(\bs) (1+|s_1|^2)^{r/2} \, \ldots \,
(1+|s_d|^2)^{r/2}
\, e^{\pi i(\bs,\cdot)} \Big\|_p
\]
is finite, where $\hat f(\bs)$ denotes the usual $\bs$th Fourier coefficient of $f$. There are some different equivalent definitions of $\Wrp$, for instance, in terms of Weil fractional derivatives (see, e.g, \cite{Di00}).

We use the notations:
$A_n(f) \ll B_n(f)$ if $A_n(f) \le CB_n(f)$ with 
$C$ an absolute constant not depending on $n$ and/or $f \in W,$ and 
$A_n(f) \asymp B_n(f)$ if $A_n(f) \ll B_n(f)$ and $B_n(f) \ll A_n(f).$
Denote by  $\lfloor y \rfloor$ the integer part of  $y \in \RR$. For a function $f \in L_p$ and a vector
$\bk \in \ZZdp$ we define the set
\[
\Pi(\bk) =\bigl\{\bs \in \ZZd: \lfloor 2^{k_i - 1} \rfloor \le |s_i| < 2^{k_i}, i \in [d] \bigr\},
\]
and the function
\begin{equation} \nonumber
\delta_\bk(f,\bx) :=
\sum_{\bs \in \Pi(\bk)}
\hat f(\bs)e^{\pi i(\bs,\bx)}.
\end{equation}

For the following lemma see \cite{NiN75} (also \cite[Chapter III, 15.2]{BeIlNi78}).

\begin{lemma} \label{lemma[LPtheorem]}
Let $1 < p < \infty$ and $r > 0$. Then we have the following norm equivalence
\begin{equation*} 
\|f\|_{\Wrp}
\ \asymp \ 
\biggl\| \biggl(\sum_{\bk \in \ZZdp} 
 \big|2^{r|\bk|_1}\delta_\bk(f) \big|^2 \biggl)^{1/2}\biggl\|_p,
\quad  \forall f \in \Wrp. 
\end{equation*}
\end{lemma}

 Put $P_\bk:= \{\bh \in \RRd:\, |h_i| \le 2^{-k_i}, \ i \in [d]\}$. 
The following lemma has been proven in \cite{U06}.

\begin{lemma} \label{normequivalence[Delta]}
Let $1 < p < \infty$ and $r < \ell$. Then we have the following norm equivalence
\begin{equation*} 
\|f\|_{\Wrp}
\ \asymp \ 
\sum_{e \subset [d]}\, \biggl\| \biggl(\sum_{\bk \in \ZZdpe} 
 \biggl(2^{(r+1)|\bk|_1} \int_{P_{\bk}}\big|\Delta_\bh^{\ell,e}(f)\big|\, \rd \bh \biggl)^2 \biggl)^{1/2}\biggl\|_p,
\quad  \forall f \in \Wrp. 
\end{equation*}
\end{lemma}

For $u \subset [d]$, let
\begin{equation*}
\omega_\ell^u(f,\bt)_p:= \sup_{h_i < t_i, i \in u}\|\Delta_\bh^{\ell,u}(f)\|_p, \ \bt \in {\II}^d,
\end{equation*} 
be the mixed $(\ell,u)$th modulus of smoothness of $f$ (in particular,
 $\omega_l^{\varnothing}(f,t)_p = \|f\|_p$).

If $0 <  p, \theta \le \infty$, 
$r> 0$ and $\ell > r$, 
we introduce the quasi-semi-norm 
$|f|_{B_{p, \theta}^{r,u}}$ for functions $f \in L_p$ by
\begin{equation*} \label{BesovSeminorm}
|f|_{B_{p, \theta}^{r,u} }:= 
\begin{cases}
 \ \left(\int_{{\II}^d} \{ \prod_{i \in u} t_i^{- r}
\omega_\ell^u(f,\bt)_p \}^ \theta \prod_{i \in u} t_i^{-1} \operatorname{d} \bt \right)^{1/\theta}, 
& \theta < \infty, \\
  \sup_{\bt \in {\II}^d} \ \prod_{i \in u} t_i^{-r}\omega_\ell^u(f,\bt)_p,  & \theta = \infty
\end{cases}
\end{equation*}
(in particular, $|f|_{B_{p, \theta}^{\alpha,\emptyset}} = \|f\|_p$).

For $0 <  p, \theta \le \infty$ and $0 < r< l,$ the Besov space 
$B_{p, \theta}^r$ is defined as the set of  functions $f \in L_p$ 
for which the Besov quasi-norm $\|f\|_{B_{p, \theta}^r}$ is finite. 
The  Besov quasi-norm is defined by
\begin{equation*} 
\|f\|_{B_{p, \theta}^r}
:= \ 
 \sum_{u \subset [d]} |f|_{B_{p, \theta}^{r,u} }.
\end{equation*}

% \begin{lemma} \label{lemma[enbeddingWtoB]}
% For $0 < p \le \infty$ and $r > 0$, there holds the embedding 
% $\Wrp \hookrightarrow B_{p, \infty}^r$ or, equivalently,
% \begin{equation*} 
% \|f\|_{B_{p, \infty}^r}
% \ \ll \ 
% \|f\|_{\Wrp},
% \quad  \forall f \in \Wrp. 
% \end{equation*}
% \end{lemma}

% \begin{proof}
% This lemma is known for $1 \le p \le \infty$. For completeness, we prove it for $0 < p <1$. Let 
% $f =  F_r\ast \varphi \in \Wrp$ with $\varphi \in L_p$. From the Minkowski inequality for $0 < p <1$
% \begin{equation*} 
% \biggl\|\int_{\TTd} g(\cdot, \by) \operatorname{d} \by \biggl\|_p
% \ \le \ 
% \biggl(\int_{\TTd} \|g(\cdot, \by) \|_p^p \operatorname{d}\by \biggl)^{1/p}
% \quad  \forall f \in \Wrp. 
% \end{equation*}
% and the inclusion $F_r \in B_{1, \infty}^r$
% (see, e.g., \cite[Chpt. III, (3.1)]{Te93}), we have for every $e \in [d]$ and $\bh \in \IId$,
% \begin{equation*} 
% \|\Delta_\bh^{\ell,e}(f)\|_p
% \ \le \
% \|\Delta_\bh^{\ell,e}(F_r)\|_p \|\varphi\|_p 
% \ \le \
% \|\Delta_\bh^{\ell,e}(F_r)\|_1 \|\varphi\|_p
% \ \ll \
% \|f\|_{\Wrp} \prod_{i \in e} h_i^r
% \end{equation*}
% which immediately implies the lemma for $0 < p <1$.
% \end{proof}

\subsection{Maximal functions}
\label{Maximal functions}

For a locally integrable function $f$ on $\TT$,  the Hardy-Littlewood maximal function is defined as
\[
M(f,x)
:= \
\sup_{h > 0} \frac{1}{2h} \int_{x-h}^{x+h} |f(y)| \, \rd y.
\]
For $i \in [d]$ a locally integrable function $f$ on $\TTd$,  the partial Hardy-Littlewood maximal function 
$M_i(f)$ is defined as the univariate maximal function in variable $x_i$ by considering $f$ as a univariate
function of $x_i$ with the other variables held fixed. The mixed Hardy-Littlewood maximal function is defined as 
\[
\bM(f)
:= \
M_d(M_{d-1}( \cdots M_1(f)\cdots )).
\]

From a result in \cite{FS71} follows 

\begin{lemma} \label{Hardy-LittlewoodM}
Let $1 < p < \infty$ and $\big(f_\bk\big)_{\bk \in \ZZdp}$ be a sequence of locally integrable functions on $\TTd$.
The we have
\[
\biggl\|\biggl(\sum_{\bk \in \ZZdp} |\bM(f_\bk)|^2\biggl)^{1/2}\biggl\|_p
\ \ll
\biggl\|\biggl(\sum_{\bk \in \ZZdp} |f_\bk|^2\biggl)^{1/2}\biggl\|_p.
\] 
\end{lemma}

Let $\bm, \bn \in \RRd$ with positive components. 
For a continuous function $f$ on $\TTd$,  the Peetre maximal function is defined as
\[
P_{\bm,\bn}(f,\bx)
:= \
\sup_{\bh \in \TTd} \frac{|f(\bx + \bh)|}{(1 + m_1|h_1|)^{n_1} \cdots  (1 + m_d|h_d|)^{n_d}}.
\]
If $\bn = (n,...,n)$, we write $P_{\bm,\bn}(f,\bx):= P_{\bm,n}(f,\bx)$.

For the proof of the following lemma see \cite[Lemma 2.3.3]{ST87} and \cite[Lemma 3.3.1]{U06}.
 
\begin{lemma} \label{Delta<P}
For the univariate trigonometric polynomial $f$ of degree $\le m$, we have 
\[
|\Delta^\ell_h(f,x)|
\ \le \
C\, \min(1,|mh|^\ell)\max(1,|mh|^n)\,P_{m,n}(f,x),
\]
where $C > 0$ is a constant independent of $f,m,h$.
\end{lemma}

For the proof of the following lemma see \cite{U06}.

\begin{lemma} \label{PeetreM}
Let $1 < p < \infty$ and $\big(f_\bk\big)_{\bk \in \ZZdp}$ be a sequence of trigonometric polynomials $f_\bk$ of 
degree $\bm^{(\bk)}$, and $\bn^{(\bk)}$ be such that $n^{(\bk)}_i > \max(\frac{1}{p},\frac{1}{2})$, $i \in [d]$. Then we have
The we have
\[
\biggl\|\biggl(\sum_{\bk \in \ZZdp} |P_{\bm^{(\bk)},\bn^{(\bk)}}(f_\bk)|^2\biggl)^{1/2}\biggl\|_p
\ \le
C\, \biggl\|\biggl(\sum_{\bk \in \ZZdp} |f_\bk|^2\biggl)^{1/2}\biggl\|_p,
\]
where $C > 0$ is a constant independent of $f$ and $\big(\bm^{(\bk)}\big)_{\bk \in \ZZdp}$. 
\end{lemma}

\subsection{Direct theorem of sampling representation}
In this subsection we prove the following direct theorem of B-spline sampling representation in Sobolev space of mixed  smoothness.
\begin{theorem} \label{DirectThm}
Let $1 < p < \infty$ and $\max(\frac{1}{p},\frac{1}{2}) < r < \ell$. Then every function $f \in \Wrp$ can be represented as the series \eqref{eq:B-splineRepresentation} converging in the norm of $\Wrp$, and there holds
the inequality
\begin{equation} \label{DirectIneq}
\biggl\| \biggl(\sum_{\bk \in \ZZdp} 
\Big|2^{r|\bk|_1} q_\bk(f)\Big|^2 \biggl)^{1/2}\biggl\|_p
\ \ll \
\|f\|_{\Wrp}. 
\end{equation}
\end{theorem}

\begin{proof} Let $f \in \Wrp$.
% Recall that 
% \[
% q_\bk(f)
% := \
% \sum_{\bs \in I(\bk)} c_{\bk,\bs}(f)\, N_{\bk,\bs}.
% \] 
Put 
\begin{equation} \label{[I^*(bk)]}
I^*(\bk):= \{\bs \in I(\bk): s_i = 0,\ell, 2\ell,...,(2^k - 1)\ell,\ i \in [d]\},
\end{equation}
 and
$I_\bn(\bk):= \bn + I^*(\bk)$ for $\bn \in \{0,..,\ell - 1\}^d$.  Then we can split $q_\bk(f)$ into a sum as
\[
q_\bk(f)
\ = \
\sum_{\bn \in \{0,..,\ell - 1\}^d}q_\bk^\bn(f)
\ = \
\sum_{\bn \in \{0,..,\ell - 1\}^d}
\sum_{\bs \in I_\bn(\bk)} c_{\bk,\bs}(f)\, N_{\bk,\bs}.
\] 
Hence,
\begin{equation*} 
\biggl\| \biggl(\sum_{\bk \in \ZZdp} 
 \Big|2^{r|\bk|_1} q_\bk(f) \Big|^2 \biggl)^{1/2}\biggl\|_p
\ \le \
\sum_{\bn \in \{0,..,\ell - 1\}^d}
\biggl\| \biggl(\sum_{\bk \in \ZZdp} 
\Big|2^{r|\bk|_1} q_\bk^\bn(f)\Big|^2 \biggl)^{1/2}\biggl\|_p
\end{equation*}
and consequently, to prove \eqref{DirectIneq} it is sufficient to show that each term in sum in the right hand side
$\ll \|f\|_{\Wrp}$. We prove this for instance, for the term corresponding to $\bn = {\bf 0}$. Let us rewrite the inequality to be proven in the following more convenient form
\begin{equation} %\label{DirectIneq-q^*}
A(f)
:= \
\biggl\| \biggl(\sum_{\bk \in \ZZdp} 
\Big|2^{r|\bk|_1} q_\bk^*(f)\Big|^2 \biggl)^{1/2}\biggl\|_p
\ = \ 
\biggl\| \biggl(\sum_{\bk \in \ZZdp} 
\biggl|2^{r|\bk|_1}\sum_{\bs \in I^*(\bk)} c_{\bk,\bs}(f)\, N_{\bk,\bs} \biggl|^2 \biggl)^{1/2}\biggl\|_p
\ \ll \
\|f\|_{\Wrp}.
\end{equation}
For a vector
$\bk \in \ZZdp$ we define the function
\begin{equation}
f_\bk :=
\begin{cases}
\delta_\bk(f), \ & \bk \in \ZZdp,
\\[1.5ex]
0, \ & \text{otherwise}.
\end{cases}
\end{equation}
Notice that since $f \in \Wrp$ with $r > \frac{1}{p}$, we can write for every $\bk \in \ZZdp$,
\[
f(\bx)
\ = \
\sum_{\bm \in \ZZd} f_{\bk + \bm}(\bx), \quad \forall \bx \in \TTd,
\] 
which yields the inequality
\begin{equation} \label{DirectIneq-q^*(1)}
\biggl\| \biggl(\sum_{\bk \in \ZZdp} 
 \biggl|2^{r|\bk|_1}\sum_{\bs \in I^*(\bk)} c_{\bk,\bs}(f)\, N_{\bk,\bs} \biggl|^2 \biggl)^{1/2}\biggl\|_p
\ \le \
\sum_{\bm \in \ZZd}
\biggl\|\biggl(\sum_{\bk \in \ZZdp} \biggl|2^{r|\bk|_1}\sum_{\bs \in I^*(\bk)} c_{\bk,\bs}(f_{\bk + \bm})\, 
N_{\bk,\bs} \biggl|^2 \biggl)^{1/2}\biggl\|_p.
\end{equation} 
We give a preliminary estimate for the terms in the right hand side
\begin{equation} \label{r}
 A_\bm(f)
 := \
\biggl\|\biggl(\sum_{\bk \in \ZZdp}\Big|2^{r|\bk|_1} q_\bk^*(f_{\bk + \bm})\, 
\Big|^2 \biggl)^{1/2}\biggl\|_p
\ = \
\biggl\|\biggl(\sum_{\bk \in \ZZdp} \biggl|2^{r|\bk|_1}\sum_{\bs \in I^*(\bk)} c_{\bk,\bs}(f_{\bk + \bm})\, 
N_{\bk,\bs} \biggl|^2 \biggl)^{1/2}\biggl\|_p.
\end{equation}
In the next step, we establish an estimate for $q_\bk^*(f_{\bk + \bm})$. 
Denote by $\sigma(\bk,\bs)$ the support of the B-spline $N_{\bk,\bs}$. Then by the construction, 
for every $\bk \in \ZZdp$, the intersection of the interiors of  $\sigma(\bk,\bs)$ and  $\sigma(\bk,\bs')$ 
is empty for different $\bs, \bs' \in I^*(\bk)$, and 
\begin{equation} \label{TTd=cup}
\TTd 
\ = \ 
\cup_{\bs \in I^*(\bk)} \sigma(\bk,\bs).
\end{equation} 
Hence, we derive that
\begin{equation} \label{q_bk^*<(1)}
|q_\bk^*(f_{\bk + \bm})(\bx)|
\ \le \ 
|c_{\bk,\bs}(f_{\bk + \bm})|
\max_{\by \in \sigma(\bk,\bs)} N_{\bk,\bs}(\by)
\ \ll \
|c_{\bk,\bs}(f_{\bk + \bm})|, \quad \forall \bx \in \sigma(\bk,\bs).
\end{equation} 
Fix a number $\nu > 0$ such that $\max(\frac{1}{p},\frac{1}{2}) < \nu < r$. From the last inequality we want to  the following inequality 
\begin{equation} \label{q_bk^*<(2)}
q_\bk^*(f_{\bk + \bm})(\bx)
\ \ll \
P_{2^{\bk + \bm}, \nu}(f_{\bk + \bm},\bx) \prod_{i=1}^{d} \min(1,2^{\ell m_i}) \max(1,2^{\nu m_i}), \quad
\forall \bx \in \TTd, \ \forall \bk \in \ZZdp,  \ \forall \bm \in \ZZd.
\end{equation}
We first obtain the univariate case of this inequality which is of the form 
\begin{equation} \label{q_k^*<(1)}
q_k^*(f_{k + m})(x)
\ \ll \
P_{2^{k + m}, \nu}(f_{k + m},x) \min(1,2^{\ell m}) \max(1,2^{\nu m}), \quad
\forall x \in \TT, \ \forall k \in \ZZ_+,  \ \forall m \in \ZZ.
\end{equation}
from the inequality for every $k \in \ZZ_+$,
\begin{equation} \label{q_k^*<(2)}
|q_k^*(f_{k + m})(x)|
\ \ll \
|c_{k,s}(f_{k + m})|, \quad \forall x \in \sigma(k,s), \ \forall s \in I^*(k).
\end{equation}
Let $x \in \TT$ and $k \in \ZZ_+$ be given.  Then by \eqref{TTd=cup} there is a $s \in I^*(k)$ such that 
$x \in \sigma(k,s)$. Notice that $|x - sh^{(k)}| \le 2^{-k}$.
If $k=0$, then by Lemma \ref{lemma[Q_bk^even&odd]} we have 
\[
c_{0,s}(f_{k + m})  = T_{h^{(0)}}^{[P_\Lambda]}(f_{k + m})(sh^{(0)}),
\] and therefore,
\[
|c_{k,s}(f_{k + m})|
\ \ll \
|f_{k + m}(sh^{(0)})|
\ \le \
\sup_{|y| \le 1}|f_{k + m}(x + y)|
\ \le \
\sup_{|y| \le 1}\frac{|f_{k + m}(x + y)|}{\big(1 + 2^k|y|\big)^\nu}
\ \le \
P_{2^k,\nu}(x).
\] 
If $k>0$, by Lemma \ref{lemma[Q_bk^even&odd]} we have 
\begin{equation} \nonumber
c_{k,s}(f) 
 \ = \
\begin{cases}
T_{h^{(k)}}^{[P_{\operatorname{even}}^*]}\circ \Delta_{h^{(k)}}^{\ell} (f)(sh^{(k)}),  & s \ \text{even} \\[1.5ex]
T_{h^{(k)}}^{[P_{\operatorname{odd}}^*]}\circ \Delta_{h^{(k)}}^{\ell} (f)(sh^{(k)}), & s \ \text{odd}.
\end{cases} 
\end{equation}
and therefore, 
\[
\begin{split}
|c_{k,s}(f_{k + m})|
\ &\ll \
|\Delta_{h^{(k)}}^{\ell} (f_{k + m})(sh^{(k)})|
\ \le \
2^\ell \sup_{|y| \le 2^{-k}}|f_{k + m}(x + y)|
\\[1ex]
\ &\ll \
\sup_{|y| \le 2^{-k}}\frac{|f_{k + m}(x + y)|}{\big(1 + 2^k|y|\big)^\nu}
\ \le \
P_{2^k,\nu}(x).
\end{split}
\]
Notice that for a continuous function $g$ and $a \ge 0$,
\begin{equation} \label{P<}
P_{2^k,\nu}(g,x)
\ \le \
2^{\nu a} P_{2^{k+a},\nu}(g,x).
\end{equation} 
All these together give 
\begin{equation} \label{c_{k,s}(m>0)}
|c_{k,s}(f_{k + m})|
\ \ll \
P_{2^k,\nu}(f_{k + m},x)
\ \le \
2^{\nu m} P_{2^{k+m},\nu}(f_{k + m},x), \quad \forall m \ge 0.
\end{equation}  
On the other hand, if $m \ge -k$ we have  by Lemma \ref{Delta<P} another estimate the trigonometric polynomial 
$f_{k + m}$ of degree $2^{k + m}$,
\begin{equation} \label{|c_{k,s}(f_{k + m})|<}
\begin{split}
|c_{k,s}(f_{k + m})|
\ &\ll \
|\Delta_{h^{(k)}}^{\ell} (f_{k + m})(sh^{(k)})|
\\[1ex]
\ &\ll \
\min(1,|2^{k + m}h^{(k)}|^\ell)\max(1,|2^{k + m}h^{(k)}|^\nu)\,P_{2^{k+m},\nu}(f_{k + m},sh^{(k)}).
\end{split}
\end{equation} 
Hence, there holds the inequality
\begin{equation} \label{c_{k,s}<(m<0)}
|c_{k,s}(f_{k + m})|
\ \ll \
2^{\ell m}\,P_{2^{k+m},\nu}(f_{k + m},x), \quad  \forall m <0.
\end{equation} 
Indeed, if  $m < 0$ and $m+k \ge 0$,
\[
P_{2^{k+m},\nu}(f_{k + m},sh^{(k)})
\ = \ 
\sup_{y \in \TT} \frac{|f_{k + m}(sh^{(k)} + y)|}{(1 + 2^{k+m}|y|)^\nu}
\ = \ 
\sup_{y \in \TT} \frac{|f_{k + m}(x + y)|}{(1 + 2^{k+m}|sh^{(k)} - x + y|)^\nu}.
\]
Since $|x - sh^{(k)}| \le 2^{-k}$, we have  for all $m < 0$,
\[
1 + 2^{k+m}|sh^{(k)} - x + y| 
\ \ge \
1 + 2^{k+m}(|y| - |sh^{(k)} - x|)
\ \ge \frac{1}{2}(1 + 2^{k+m+1}(|y|),
\]
and consequently, by \eqref{P<},
\[
P_{2^{k+m},\nu}(f_{k + m},sh^{(k)})
\ \le \
2 \sup_{y \in \TT} \frac{|f_{k + m}(x + y)|}{(1 + 2^{k+m+1}|y|)^\nu}
\ \le \
4 P_{2^{k+m},\nu}(f_{k + m},x)
\]
which together with the equation $h^{(k)} = \ell^{-1}2^{-k}$ and 
\eqref{|c_{k,s}(f_{k + m})|<} proves \eqref{c_{k,s}<(m<0)} $m < 0$ and $m+k \ge 0$. 
In the  case $m < 0$ and $m+k < 0$,  \eqref{c_{k,s}<(m<0)} is trivial because by definition $f_{k + m}= 0$. By combining 
\eqref{q_k^*<(2)}, \eqref{c_{k,s}(m>0)} and \eqref{c_{k,s}(m>0)} we prove \eqref{q_k^*<(1)}.
The $d$-variate inequality \eqref{q_bk^*<(1)} can be easily derived from the univariate inequality 
\eqref{q_k^*<(1)} by a tensor product argument.

We are now in position to estimate $A_\bm(f)$. 
Indeed, putting for $\bm \in \ZZd$,
\begin{equation} \label{b_bm}
b_\bm:= \ \prod_{i=1}^d b_{m_i}, \quad 
b_{m_i}
:= \
\begin{cases}
2^{(\ell - r)m_i}, \ & \text{if} \ m_i < 0; \\
2^{(\nu - r)m_i}, \ & \text{if} \ m_i \ge 0,
\end{cases}
\end{equation}
from  \eqref{DirectIneq-q^*(1)} and \eqref{q_k^*<(1)} and 
Lemmas  \ref{PeetreM} and \ref{lemma[LPtheorem]} it follows that
\begin{equation} \label{}
\begin{split}
 A_\bm(f)
 \ &\ll \
\biggl\|\biggl(\sum_{\bk \in \ZZdp}\Big|2^{r|\bk + \bm|_1} b_\bm P_{2^{\bk + \bm}, \nu}(f_{\bk + \bm}) 
 \Big|^2 \biggl)^{1/2}\biggl\|_p
 \\[1ex]
\ &\ll \
b_\bm \biggl\|\biggl(\sum_{\bk \in \ZZdp}\Big|2^{r|\bk + \bm|_1} P_{2^{\bk + \bm}, \nu}(f_{\bk + \bm}) 
 \Big|^2 \biggl)^{1/2}\biggl\|_p
  \\[1ex]
\ &\ll \
b_\bm \biggl\|\biggl(\sum_{\bk \in \ZZdp}\Big|2^{r|\bk + \bm|_1} (f_{\bk + \bm}) 
 \Big|^2 \biggl)^{1/2}\biggl\|_p
  \\[1ex]
\ &\le \
b_\bm \biggl\|\biggl(\sum_{\bk \in \ZZdp}\Big|2^{r|\bk|_1} f_{\bk} 
 \Big|^2 \biggl)^{1/2}\biggl\|_p
 \ \asymp \ b_\bm \|f\|_{\Wrp}.
\end{split}
\end{equation}
Hence, by  \eqref{DirectIneq-q^*(1)}, \eqref{b_bm} and the inequalities $\ell - r > 0$ and $\nu - r < 0$, we have
\begin{equation} \label{}
\begin{split}
A(f)
\ \le \
\sum_{\bm \in \ZZd} A_\bm(f)
 \ \ll \ 
\|f\|_{\Wrp} \sum_{\bm \in \ZZd} b_\bm 
 \ \ll \ 
\|f\|_{\Wrp}.
\end{split}
\end{equation}
The proof is complete.
\end{proof}

Theorem \ref{DirectThm} has been proven in \cite{ByUl16} for the case $\ell =2$ in terms of Faber series (see Subsection \ref{Examples}).

\subsection{Inverse theorem of sampling representation}

\begin{theorem} \label{InverseThm}
Let $1 < p < \infty$ and $0 < r < \ell - 1$.
Then for every function $f$ on $\TTd$  represented as a B-spline series 
\begin{equation} \label{InverseThm:Representation}
f \ = \sum_{\bk \in {\ZZ}^d_+} \ q_\bk = 
\sum_{\bk \in {\ZZ}^d_+} \sum_{\bs \in I(\bk)} c_{\bk,\bs}N_{\bk,\bs}, 
\end{equation}  
we have $f \in \Wrp$ and
\begin{equation} \label{InverseIneq}
\|f\|_{\Wrp}
\ \ll \
\biggl\| \biggl(\sum_{\bk \in \ZZdp} 
\Big|2^{r|\bk|_1} q_\bk\Big|^2 \biggl)^{1/2}\biggl\|_p,
\end{equation}
whenever the right hand side is finite.
\end{theorem}

\begin{proof} 
For $\bk \in \ZZdp$, let $I^*(\bk)$ be the subset in $I(\bk)$ defined in the proof of Theorem \ref{DirectThm}.
By the same argument as in the proof of Theorem \ref{DirectThm} it is sufficient to prove that for a 
function $f$ on $\TTd$  represented as a B-spline series 
\begin{equation} \label{InverseThm:Representation^*}
f \ = \sum_{\bk \in {\ZZ}^d_+} \ q_\bk^* = 
\sum_{\bk \in {\ZZ}^d_+} \sum_{\bs \in I^*(\bk)} c_{\bk,\bs}N_{\bk,\bs}, 
\end{equation}  
we have
\begin{equation} \label{InverseIneq^*(1)}
\|f\|_{\Wrp}
\ \ll \
\biggl\| \biggl(\sum_{\bk \in \ZZdp} 
\big|2^{r|\bk|_1} q_\bk^*\big|^2 \biggl)^{1/2}\biggl\|_p,
\end{equation}
whenever the right hand side is finite. Due to Lemma \ref{normequivalence[Delta]} 
the last inequality is equivalent to
\begin{equation} \label{InverseIneq-q^*(2)}
\biggl\| \biggl(\sum_{\bk \in \ZZdpe} 
 \biggl(2^{(r+1)|\bk|_1} \int_{P_{\bk}}\big|\Delta_\bh^{\ell - 1,e}(f)\big|\, \rd \bh \biggl)^2 \biggl)^{1/2}\biggl\|_p
\ \ll \
\biggl\| \biggl(\sum_{\bk \in \ZZdp} 
\big|2^{r|\bk|_1} q_\bk^*\big|^2 \biggl)^{1/2}\biggl\|_p,
\quad
\forall e \subset [d].
\end{equation}
Let us verify this inequality for the case $e = [d]$  what is 
\begin{equation} \label{InverseIneq-q^*()}
B(f)
:=  \
\biggl\| \biggl(\sum_{\bk \in \ZZdp} 
 \biggl(2^{(r+1)|\bk|_1} \int_{P_{\bk}}\big|\Delta_\bh^{\ell - 1}(f)\big|\, \rd \bh \biggl)^2 \biggl)^{1/2}\biggl\|_p
\ \ll \
\biggl\| \biggl(\sum_{\bk \in \ZZdp} 
\big|2^{r|\bk|_1} q_\bk^*\big|^2 \biggl)^{1/2}\biggl\|_p.
\end{equation}
The case  $e \not= [d]$ can be proven similarly with a slight modification.
For $\bk \in \ZZd \setminus \ZZdp$, we introduce the convention: $I^*(\bk):= \varnothing$, $q_\bk^*:=0$,
$c_{\bk,\bs}:=0$, $N_{\bk,\bs}:=0$ for $\bs \in I^*(\bk)$. With this convention we can write 
\begin{equation} \nonumber
q_\bk^* 
\ = \ 
\sum_{\bs \in I^*(\bk)} c_{\bk,\bs}N_{\bk,\bs}, \quad \forall \bk \in \ZZd.
\end{equation}  
Put $\ZZdu:= \{\bm \in \ZZd:\, m_i > 0, \ i \in u, \ m_i \le 0, \ i \not\in u \}$ for $u \subset [d]$.
Notice that we have for every $\bk \in \ZZdp$,
\[
f(\bx)
\ = \
\sum_{u \subset [d]} \sum_{\bm \in \ZZdu} q_{\bk + \bm}^*(\bx), \quad \forall \bx \in \TTd,
\] 
which yields the inequality
\begin{equation} \label{InverseIneq-q^*(1)}
B(f)
\ \le \
\sum_{u \subset [d]}\sum_{\bm \in \ZZdu} B_\bm(f),
\end{equation}
where 
\begin{equation*}
B_\bm(f)
:= \
\biggl\| \biggl(\sum_{\bk \in \ZZdp} 
 \biggl(2^{(r+1)|\bk|_1} \int_{P_{\bk}}\big|\Delta_\bh^{\ell - 1}(q_{\bk+\bm}^*)\big|\, \rd \bh \biggl)^2 \biggl)^{1/2}\biggl\|_p.
\end{equation*} 
For a given $x \in \TT$, we preliminarily estimate the univariate integral  
\[
2^k\int_{P_{k}}\big|\Delta_h^{\ell - 1}(q_{k+m}^*,x)\big|\,\rd h.
\]
Let $I^*(k+m;x)$ be the subset in $I^*(k+m)$ of all $s$ such that $|x - s2^{-k-m}| \le \ell 2^{-k}$ if $k+m \ge 0$, and  $I^*(k+m;x) = \varnothing$ if $k+m < 0$. Then from the equation
\begin{equation} \nonumber
\Delta_h^{\ell - 1}(q_{k+m}^*,x) 
\ = \
\sum_{s \in I^*(k+m;x)} c_{k+m,s}\,\Delta_h^{\ell - 1}(N_{k+m,s},x), \quad |h| \le 2^{-k},
\end{equation}
we have
\begin{equation} \nonumber
2^k\int_{P_{k}}\big|\Delta_h^{\ell - 1}(q_{k+m}^*,x)\big|\,\rd h 
\ \le \
\sum_{s \in I^*(k+m;x)} |c_{k+m,s}|\,2^k \int_{|h| \le 2^{-k}}\big|\Delta_h^{\ell - 1}(N_{k+m,s},x)\big|\,\rd h 
\end{equation}
If $m > 0$, from the definition \eqref{def[Delta_h^ell]} we derive that
\begin{equation} \nonumber
2^k \int_{|h| \le 2^{-k}}\big|\Delta_h^{\ell - 1}(N_{k+m,s},x)\big|\,\rd h 
\ \le \
\sum_{j=0}^{\ell - 1} \binom{\ell - 1}{j} 2^k \int_{|h| \le 2^{-k}}\big|N_{k+m,s}(x + jh)\big|\,\rd h 
\ \ll \
N_{k+m,s}(x) + 2^{-m}.
\end{equation} 
Notice that for $k+m \ge 0$, the B-splines $N_{k+m,s}$ have the $\ell - 1$ derivative uniformly bounded by 
$C2^{(\ell - 1)(k+m)}$ with an absolute constant $C$. Hence, we get for $|h| \le 2^{-k}$,
\begin{equation} \nonumber
\big|\Delta_h^{\ell - 1}(N_{k+m,s},x)\big|
\ \ll \ 
|h|^{\ell - 1} 2^{(\ell - 1)(k+m)}
\ \ll \
2^{(\ell - 1)m},
\end{equation}
and consequently,
\begin{equation} \nonumber
2^k \int_{|h| \le 2^{-k}}\big|\Delta_h^{\ell - 1}(N_{k+m,s},x)\big|\,\rd h 
\ \ll \
2^{(\ell - 1)m}.
\end{equation}  
Taking account that $q_{k+m}^* = 0$ for $k+m < 0$, and summing up we arrive at the estimate
\begin{equation} \label{ineq[intDq]}
2^k\int_{P_{k}}\big|\Delta_h^{\ell - 1}(q_{k+m}^*,x)\big|\,\rd h 
\ \ll \
\begin{cases}
\sum_{s \in I^*(k+m;x)} |c_{k+m,s}| \big(N_{k+m,s}(x) + 2^{-m}\big), \ &  m \ge 0; \\[1ex] 
\sum_{s \in I^*(k+m;x)} |c_{k+m,s}| 2^{(\ell - 1)m}, \ &  m < 0.
\end{cases} 
\end{equation}
We introduce some notations: for $u \subset [d]$ and $\bx \in \RRd$
 $\bar{u}:= [d] \setminus u$, $\bx(u)$ is a the element in $\RRd$ such that  $x(u)_i = x_i$ if $i \in u$ and
$x(u)_i = 0$ otherwise.  Then for a given $\bx \in \TTd$, from  \eqref{ineq[intDq]} by a tensor product argument we obtain for every $\bk \in \ZZdp$ and every $\bm \in \ZZdpu$,
\begin{equation} \label{ineq[g_{bk+bm}(bx)]}
\begin{split}
g_{\bk+\bm}(\bx)
&:= \
2^{|\bk|_1}\int_{P_{\bk}}\big|\Delta_{\bh}^{\ell - 1}(q_{\bk+\bm}^*,\bx)\big|\,\rd \bh
\\[1ex] 
\ &\ll \
\sum_{\bs \in I^*(\bk;\bx)} |c_{\bk,\bs}| 2^{(\ell - 1)|\bm(\bar{u})|_1} \prod_{i \in u}N_{k_i,s_i}(x_i) 
+ \sum_{\bs \in I^*(\bk;\bx)} |c_{\bk,\bs}| 2^{(\ell - 1)|\bm(\bar{u})|_1}  2^{-|\bm(u)|_1}, 
\end{split}
\end{equation}
where $I^*(\bk;\bx):= \prod_{i=1}^d I^*(k_i;x_i)$.
By using the inequality for the univariate periodic B-splines 
\begin{equation} \nonumber
2^k \int_{\TT} N_{k,s}(y)\,\rd y  
\ \ge \
C
\end{equation}
with an absolute constant $C > 0$,
we can continue the estimation \eqref{ineq[g_{bk+bm}(bx)]} for every $\bk \in \ZZdp$ and every $\bm \in \ZZdpu$ as
\begin{equation} \label{ineq[g_{bk+bm}(bx)](2)}
\begin{split}
g_{\bk+\bm}(\bx)
\ &\ll \
\sum_{\bs \in I^*(\bk;\bx)} |c_{\bk,\bs}| 2^{|\bk(\bar{u})|_1 + \ell|\bm(\bar{u})|_1} 
\biggl(\prod_{i \in u}N_{k_i,s_i}(x_i) \biggl)
\biggl(\prod_{i \in \bar{u}}\int_{\TT} N_{k_i,s_i}(x_i + h_i)\,\rd h_i\biggl)  
\\[1ex] 
&+ \sum_{\bs \in I^*(\bk;\bx)} |c_{\bk,\bs}| 
2^{|\bk(u)|_1} 
\biggl(\prod_{i \in u}\int_{\TT} N_{k_i,s_i}(x_i + h_i)\,\rd h_i\biggl) 
2^{|\bk(\bar{u})|_1 + \ell|\bm(\bar{u})|_1} 
\biggl(\prod_{i \in \bar{u}}\int_{\TT} N_{k_i,s_i}(x_i + h_i)\,\rd h_i\biggl)  
\\[1ex] 
\ &= \
 2^{|\bk(\bar{u})|_1 + \ell|\bm(\bar{u})|_1} 
\int_{\TTd} \biggl|\sum_{\bs \in I^*(\bk;\bx)} c_{\bk,\bs}N_{\bk+\bm,\bs}(\bx + \bh(\bar{u}))\biggl|\,\rd \bh
\\[1ex] 
&+ 2^{|\bk(u)|_1}2^{|\bk(\bar{u})|_1 + \ell|\bm(\bar{u})|_1} 
\int_{\TTd} \biggl|\sum_{\bs \in I^*(\bk;\bx)} c_{\bk,\bs}N_{\bk+\bm,\bs}(\bx + \bh)\biggl|\,\rd \bh.
\end{split}
\end{equation}
For a fixed $x \in \TT$ consider the univariate function on $\TT$ in variable $y$
\begin{equation} \nonumber
G^x_{k+m}(y)
:= \
\sum_{s \in I^*(k+m;x)} c_{k+m,s}N_{k+m,s}(y), \quad k \in \ZZ_+, \ m \in \ZZ.
\end{equation}
By the construction of the set $I^*(k+m;x)$ and the inequality  $|\supp(N_{k+m,s})| \le 2^{-k-m}$, we have
\begin{equation} \label{supportG}
|\supp(G^x_{k+m})|
\ \ll \
\begin{cases}
2^{-k}, \ &  m \ge 0; \\[1ex] 
2^{-k-m}, \ &  m < 0.
\end{cases} 
\end{equation}
Hence, by the definition of the Hardy-Littlewood maximal function we receive for every $y \in \TT$,
\begin{equation} \label{intG}
\int_{\TT} |G^x_{k+m}(y + h)| \,\rd h
\ \ll \
\begin{cases}
2^{-k} M(G^x_{k+m}(y)), \ &  m \ge 0; \\[1ex] 
2^{-k-m} M(G^x_{k+m}(y)), \ &  m < 0.
\end{cases}  
\end{equation}
Let $u \subset [d]$ and $\bx \in \TTd$ be given. We consider the  function on $\TTd$ in variable $\by$
\begin{equation} \nonumber
G^\bx_{\bk+\bm}(\by)
:= \
\sum_{\bs \in I^*(\bk+\bm;\bx)} c_{\bk+\bm,s}N_{\bk+\bm,\bs}(\by), \quad \bk \in \ZZdp, \ \bm \in \ZZdpu.
\end{equation}
From \eqref{intG} by a tensor product argument we can show that for every $\by \in \TTd$, every $\bk \in \ZZdp$ and every $\bm \in \ZZdpu$,
\begin{equation} \label{intGb}
\int_{\TTd} |G^\bx_{\bk+\bm}(\by + \bh)| \,\rd \bh
\ \ll \
2^{-|\bk(u)|_1 - |\bk(\bar{u})- |\bm(\bar{u})|_1} \bM(G^\bx_{\bk+\bm}(\by)). 
\end{equation}
and 
\begin{equation} \label{intGbu}
\int_{\TTd} |G^\bx_{\bk+\bm}(\by + \bh(\bar{u}))| \,\rd \bh
\ \ll \
2^{- |\bk(\bar{u})- |\bm(\bar{u})|_1} \bM(G^\bx_{\bk+\bm}(\by)). 
\end{equation}
Applying these inequalities for $\by = \bx$ to the right hand side in \eqref{ineq[g_{bk+bm}(bx)](2)}, 
by the equation $G^\bx_{\bk+\bm}(\bx) = q_{\bk+\bm}^*(\bx)$ we arrive at
\begin{equation} \label{ineq[g_{bk+bm}(bx)](3)}
g_{\bk+\bm}(\bx)
\ \ll \
2^{-(\ell - 1)|\bm(\bar{u})|_1} \bM(G^\bx_{\bk+\bm}(\bx))
\ = \ 
2^{-(\ell - 1)|\bm(\bar{u})|_1} \bM(q_{\bk+\bm}^*(\bx))
\end{equation}
which by Lemma \ref{Hardy-LittlewoodM} yields for every $\bm \in \ZZdpu$,
\begin{equation*}
\begin{split}
B_\bm(f)
\ &= \
\biggl\| \biggl(\sum_{\bk \in \ZZdp} 
 \Big(2^{r|\bk|_1} g_{\bk+\bm}\Big)^2 \biggl)^{1/2}\biggl\|_p
\\[1ex] 
\ &\ll \
2^{(\ell - 1 - r)|\bm(\bar{u})|_1 - |\bm(u)|_1} 
\biggl\| \biggl(\sum_{\bk \in \ZZdp} 
\Big(\bM\big(2^{r|\bk + \bm|_1} q_{\bk+\bm}^*\big)\Big)^2 \biggl)^{1/2}\biggl\|_p
\\[1ex] 
\ &\ll \
2^{(\ell - 1 - r)|\bm(\bar{u})|_1 - |\bm(u)|_1} 
\biggl\| \biggl(\sum_{\bk \in \ZZdp} 
\Big(\big|2^{r|\bk + \bm|_1} q_{\bk+\bm}^*\big|^2 \biggl)^{1/2}\biggl\|_p
\\[1ex] 
\ &\le \
2^{(\ell - 1 - r)|\bm(\bar{u})|_1 - |\bm(u)|_1} 
\biggl\| \biggl(\sum_{\bk \in \ZZdp} 
\big|2^{r|\bk|_1} q_{\bk}^*\big|^2 \biggl)^{1/2}\biggl\|_p.
\end{split}  
\end{equation*}
From the last inequality 
% By replacing $B_\bm(f)$ with the right hand side in \eqref{InverseIneq-q^*(1)} and 
taking account of the
inequality $\ell - 1 - r > 0$, we obtain
\begin{equation} \nonumber
B(f)
\ \ll \
\sum_{u \subset [d]}\sum_{\bm \in \ZZdu} 2^{(\ell - 1 - r)|\bm(\bar{u})|_1 - |\bm(u)|_1} 
\biggl\| \biggl(\sum_{\bk \in \ZZdp} 
\big|2^{r|\bk|_1} q_{\bk}^*\big|^2 \biggl)^{1/2}\biggl\|_p
\ \ll \
\biggl\| \biggl(\sum_{\bk \in \ZZdp} \big|2^{r|\bk|_1} q_{\bk}^*\big|^2 \biggl)^{1/2}\biggl\|_p
\end{equation} 
which proves \eqref{InverseIneq-q^*()} and therefore, the theorem.
\end{proof}

Theorem \ref{InverseThm} has been proven in \cite{ByUl16} for the case $\ell =2$ in terms of Faber series (see Subsection \ref{Examples}).

\begin{corollary}
Let $1 < p < \infty$ and $\max(\frac{1}{p},\frac{1}{2}) < r < \ell - 1$.
Then we have
\begin{equation} \nonumber
\biggl\| \biggl(\sum_{\bk \in \ZZdp} 
\Big|2^{r|\bk|_1} q_\bk(f)\Big|^2 \biggl)^{1/2}\biggl\|_p
\ \asymp \
\|f\|_{\Wrp},
\quad  \forall f \in \Wrp. 
\end{equation}
\end{corollary}

\subsection{Theorems of sampling representation in Besov spaces}

Theorems on B-spline quasi-interpolation sampling representations with discrete equivalent quasi-norm in terms of coefficient functionals have been proved in \cite{Di09}--\cite{Di13}, \cite{DU14} for various non-periodic  Besov spaces. Let us now state direct and inverse theorems on a quasi-interpolation representation  in periodic spaces $B_{p, \theta}^r$ by the B-splines series \eqref{eq:B-splineRepresentation}, which can be proven in the same way as for
non-periodic  Besov spaces.

\begin{theorem}
Let $0< p,\theta \leq \infty$ and $1/p<r<2\ell$. Then every function $f \in B_{p, \theta}^r$ can be represented as the series \eqref{eq:B-splineRepresentation} converging in the norm of $B_{p, \theta}^r$, and there holds
the inequality
\begin{equation}\nonumber
\biggl( \sum_{\bk \in \ZZdp} \
2^{r|\bk|_1\theta}\|q_\bk(f)\|^{\theta}_p\biggl)^{1/\theta} 
\ \ll \
\|f\|_{B_{p, \theta}^r}
\end{equation}
for all $f\in B^r_{p,\theta}$, with the sum over $\bk$ changing to the supremum when $\theta = \infty$.
\end{theorem}

\begin{theorem} 
Let $0 < p,\theta \leq \infty$ and $0 < r < \min\{2\ell, 2\ell - 1 + 1/p\}$.
Then for every function $f$ on $\TTd$  represented as a B-spline series 
\eqref{InverseThm:Representation}
belongs to $B_{p, \theta}^r$ and 
\begin{equation}\nonumber
    \|f\|_{B_{p, \theta}^r} 
\ \ll \
 \biggl( \sum_{\bk \in \ZZdp} \
2^{r|\bk|_1\theta}\|q_\bk\|^{\theta}_p\biggl)^{1/\theta}
\end{equation}
with the sum over $\bk$ changing to the supremum when $\theta = \infty$, whenever the right hand side is finite.
\end{theorem}

\begin{corollary} 
Let $0 < p, \theta \le \infty$ and  $1/p < r < \min\{2\ell, 2\ell - 1 + 1/p\}$. 
Then a  periodic function $f \in B_{p, \theta}^r$ can be represented
by the B-spline series \eqref{eq:B-splineRepresentation}
satisfying the relation
\begin{equation} \nonumber
\biggl( \sum_{\bk \in \ZZdp} \
2^{r|\bk|_1\theta}\|q_\bk(f)\|^{\theta}_p\biggl)^{1/\theta} 
\ \asymp \ 
\|f\|_{B_{p, \theta}^r}
\end{equation}
with the sum over $\bk$ changing to the supremum when $\theta = \infty$.
\end{corollary}

\section{Sampling recovery}
\label{Sampling recovery}

For $m \in {\ZZ}_+$, we define the operator $R_m$ by 
\begin{equation*}
R_m(f) 
:= \ 
\sum_{|\bk|_1 \le m} q_\bk(f)
\ = \
\sum_{|\bk|_1 \le m} \ \sum_{\bs \in I(\bk)} c_{\bk,\bs}(f)\, N_{\bk,\bs}.
\end{equation*} 
For functions $f$ on $\TTd$, $R_m$ defines the linear sampling algorithm 
on the Smolyak grid $G^d(m)$ 
\begin{equation*} 
R_m(f) 
\ = \ 
S_n(\bY_n,\Phi_n,f) 
\ = \ 
\sum_{\by \in G^d(m)} f(\by) \psi_{\by}, 
\end{equation*} 
where $n := \ |G^d(m)|$, $\bY_n:= \{\by \in G^d(m)\}$, $\Phi_n:= \{\varphi_{\by}\}_{\by \in G^d(m)}$  and for 
$\by =  2^{-\bk} \bs$,  $\varphi_\by$ are explicitly constructed as linear combinations of at most 
at most $m_0$ B-splines $N_{\bk,\bj}$ for some $m_0 \in \NN$ which is independent of $\bk,\bs,m$ and $f$.  

\begin{theorem} \label{thm[|f -  R_m(f)|_p<]}
Let $1 < p,q < \infty$ and $\max(\frac{1}{p},\frac{1}{2}) < r < \ell$.
Then we have
\begin{equation} %\label{[|f -  R_m(f)|_p<]}
\big\|f -  R_m(f) \big\|_q
\ \ll \
\|f\|_{\Wrp} \times
\begin{cases}
2^{-rm} m^{(d-1)/2}, \ & p \ge q, \\
2^{-(r - 1/p + 1/q)m} \, \ & p < q,
\end{cases}
\quad  \forall f \in \Wrp. 
\end{equation}
\end{theorem}

\begin{proof} Let $f$ be a function in $\Wrp$, since $r > \frac{1}{p}$, $f$ is continuous on $\TTd$ and consequently, we obtain by Lemma~\ref{lemma[representation]}
\begin{equation} \label{remainder}
f -  R_m(f) \ = \sum_{|\bk|_1 > m} \ q_\bk(f) 
\end{equation}
with uniform convergence.

We first consider the case $p \ge q$. Due to the inequality  $\|f\|_q  \le \|f\|_p $, it is sufficient prove this case of the theorem for $p=q$. From \eqref{remainder} and the H\"older inequality and Theorem  \ref{DirectThm} we have
\begin{equation*}
\begin{split}
\big\|f -  R_m(f) \big\|_p
\ &= \
\biggl\|\sum_{|\bk|_1 > m} \ q_\bk(f) \biggl\|_p
\ \le \
\biggl\|\biggl(\sum_{|\bk|_1 > m} 2^{-2r|\bk|_1} \biggl)^{1/2} 
\biggl(\sum_{|\bk|_1 > m} \big|2^{r|\bk|_1}  q_\bk(f) \big|^2 \biggl)^{1/2}\biggl\|_p
\\[1ex] 
\ &\le \
\biggl(\sum_{|\bk|_1 > m} 2^{-2r|\bk|_1} \biggl)^{1/2} 
\biggl\|\biggl(\sum_{\bk \in \ZZdp} \big|2^{r|\bk|_1}  q_\bk(f) \big|^2 \biggl)^{1/2}\biggl\|_p
\ \ll \
2^{-rm} m^{(d-1)/2}\,\|f\|_{\Wrp}.
\end{split}  
\end{equation*}
We next consider the case $p < q$. From \cite[Lemma 3]{Be74} one can prove the inequality
\begin{equation} \label{ineq[for p<q]}
\|f\|_q
\ \ll \
\|f\|_{W^{1/p - 1/q}_p}, \quad \forall f \in W^{1/p - 1/q}_p.
\end{equation}
Hence, by \eqref{remainder}, Theorems  \ref{InverseThm} and \ref{DirectThm} we derive that
\begin{equation*}
\begin{split}
\big\|f -  R_m(f) \big\|_q
\ &\ll \
\biggl\|\sum_{|\bk|_1 > m} \ q_\bk(f) \biggl\|_{W^{1/p - 1/q}_p}
\ \ll \
\biggl\|\biggl(\sum_{|\bk|_1 > m} \big|2^{(1/p - 1/q)|\bk|_1}  q_\bk(f) \big|^2 \biggl)^{1/2}\biggl\|_p
\\[1ex] 
\ &\le \
2^{-(r - 1/p + 1/q)m}
\biggl\|\biggl(\sum_{\bk \in \ZZdp} \big|2^{r|\bk|_1}  q_\bk(f) \big|^2 \biggl)^{1/2}\biggl\|_p
\ \ll \
2^{-(r - 1/p + 1/q)m}\,\|f\|_{\Wrp}.
\end{split}  
\end{equation*}
The theorem is completely proven.
\end{proof}

Denote by $\Urp$ the unit ball in the space $\Wrp$.

\begin{corollary} \label{corollary[r_n<,p<q]}
Let $1 < p,q < \infty$ and $r > \max(\frac{1}{p},\frac{1}{2})$.
Then we have
\begin{equation} \label{r_n<]}
r_n(\Urp,L_q) 
\ \ll \
\begin{cases}
\biggl(\frac{(\log n)^{d-1}}{n}\biggl)^{r} (\log n)^{(d-1)/2}, \ & p \ge q, \\[1.5ex]
\biggl(\frac{(\log n)^{d-1}}{n}\biggl)^{(r-1/p+1/q)},  \, \ & p < q.
\end{cases}
\end{equation}
\end{corollary}

\begin{proof}
 This corollary \eqref{r_n<]} is directly derived from  \eqref{thm[|f -  R_m(f)|_p<]} by considering special values of $n = |G^d(m)| \asymp 2^m m^{d-1}$. 
\end{proof}

\begin{corollary} \label{cor[r_n^s><]}
Let $1 < p,q < \infty$ and $r > \max(\frac{1}{p},\frac{1}{2})$.
Then we have
\begin{equation} \label{r_n^s><]}
r_n^s(\Urp,L_q) 
\ \asymp \
\begin{cases}
\biggl(\frac{(\log n)^{d-1}}{n}\biggl)^{r} (\log n)^{(d-1)/2}, \ & p \ge q, \\[1.5ex]
\biggl(\frac{(\log n)^{d-1}}{n}\biggl)^{(r-1/p+1/q)},  \, \ & p < q.
\end{cases}
\end{equation}
\end{corollary}

\begin{proof} We fix an even number $\ell = 2^\nu$ for some $\nu \in \NN$ such that $r < \ell - 1$.
Consider the operator $R_{m+d\nu}$ constructed on the B-splines of order $\ell$. It is a sampling algorithm on the grid $G^d(m)$. The upper bounds of \eqref{r_n^s><]} is directly derived from  \eqref{thm[|f -  R_m(f)|_p<]} and the relations $n \asymp |G^d(m)| \asymp 2^m m^{d-1}$ for the largest $m$ such that $|G^d(m)| \le n$.

To prove the lower bounds, based on the obvious inequality
\begin{equation} \label{[r^s_n>]}
r^s_n(\Urp,L_q) 
\ \ge \ \inf_{|G^d(m)| \le n} \
\sup_{f \in \Urp: \ f(\by) = 0, \ \by \in G^d(m)} \, \|f \|_q,
\end{equation} 
we will construct test functions $g \in \Urp$ with $g(\by) = 0, \ \forall \by \in G^d(m)$, and then estimate from below the norm $\|g\|_q$. Take the index set $I^*(\bk)$ as in \eqref{[I^*(bk)]} and consider the test function
\begin{equation} \nonumber
g_1
\ := \
C_1 2^{- rm} m^{-(d-1)/2}  \sum_{|\bk|_1 = m} \sum_{\bs \in \in I^*(\bk)} N_{\bk,\bs}
\end{equation}
with a constant $C_1$. Here $N_{\bk,\bs}$ are the $d$-variate periodic $B$-splines of $\ell$.
By the construction one can verify that  $g_1(\by) = 0, \ \forall \by \in G^d(m)$.
By applying Theorem \ref{InverseThm} and the inequality 
\begin{equation} \label{g_1}
\biggl|\sum_{\bs \in \in I^*(\bk)} N_{\bk,\bs}(\bx)\biggl|
\ \le \
1, \quad \forall \bx \in \TTd,
\end{equation}
we can see that $g_1 \in \Urp$ for some properly chosen value of $C_1$. Hence, by \eqref{[r^s_n>]}
\begin{equation} \nonumber
\begin{split}
r^s_n(\Urp,L_q) 
\  &\ge \ \|g_1\|_q \ \ge \|g_1\|_1
\ \gg \
  2^{-rm} m^{(d-1)/2} \\
\ &\asymp \
\biggl(\frac{(\log n)^{d-1}}{n}\biggl)^{r} (\log n)^{(d-1)/2}.
\end{split}
\end{equation}
This proves the lower bound of the case $p \ge q$. For the case $p < q$, we take a $\bk^*$ with 
$|\bk^*|_1 = m$ and a $\bs^* \in I^*(\bk)$, and consider the test function
\begin{equation} \nonumber
g_2
\ := \
C_2 2^{- (r-1/p)m} N_{\bk^*,\bs^*}
\end{equation}
with a constant $C_2$. Similarly to the function $g_1$, we have $g_2(\by) = 0, \ \forall \by \in G^d(m)$,
and $g_2 \in \Urp$ for some properly chosen value of $C_2$. Hence, by \eqref{[r^s_n>]} we obtain
\begin{equation*}
r^s_n(\Urp,L_q) 
\  \ge \ \|g_2\|_q
\  \gg \
 2^{-(r - 1/p + 1/q)m}
 \ \asymp \
\biggl(\frac{(\log n)^{d-1}}{n}\biggl)^{(r-1/p+1/q)}
\end{equation*}
which proves the lower bound of the case $p < q$.
\end{proof}

From Theorem~\ref{thm[|f -  R_m(f)|_p<]} and the proof of the lower bounds in Corollary~\ref{cor[r_n^s><]}
we also obtain

\begin{corollary} \label{cor[|f -  R_m(f)|_p><]}
Let $1 < p,q < \infty$ and $\max(\frac{1}{p},\frac{1}{2}) < r < \ell$.
Then we have
\begin{equation} \label{[|f -  R_m(f)|_p<]}
\sup_{f \in \Urp} \big\|f -  R_m(f) \big\|_q
\ \asymp \
\begin{cases}
2^{-rm} m^{(d-1)/2}, \ & p \ge q, \\
2^{-(r - 1/p + 1/q)m} \, \ & p < q.
\end{cases}
\end{equation}
\end{corollary}

\begin{theorem} \label{theorem[|f -  R_m(f)|_infty><]}
Let $1 < p < \infty$ and $1/p < r < \ell$.
Then we have
\begin{equation} \label{[|f -  R_m(f)|_infty><]}
\sup_{f \in \Urp} \big\|f -  R_m(f) \big\|_\infty
\ \asymp \
2^{-(r-1/p)m} m^{(d-1)(1-1/p)}.
\end{equation}
\end{theorem}

\begin{proof} The upper bound follows from the embedding 
$\Wrp \hookrightarrow B^{r-1/p}_{\infty,p}$ \cite{Ja77} (see also \cite[Lemma 3.7]{DTU16})
and the estimate
\begin{equation}\nonumber
\sup_{f \in  U^{r-1/p}_{\infty,p}} \big\|f -  R_m(f) \big\|_\infty
\ \ll \
2^{-(r-1/p)m} m^{(d-1)(1-1/p)}
\end{equation}
proven in \cite[Theorem 3.1]{Di11}, where $U^{r-1/p}_{\infty,p}$ is the unit ball in 
$B^{r-1/p}_{\infty,p}$.
The lower bound can be proven in a similar way to that of  \cite[Theorem 2.1]{Te93}. 
\end{proof}

\begin{corollary} \label{corollary[r_n,p<q]}
Let $1 < p < q \le 2$  or $2\le p < q < \infty$ and $r > \max(\frac{1}{p},\frac{1}{2})$.
Then we have
\begin{equation}  \label{[r_n,p<q]}
r_n(\Urp,L_q) 
\ \asymp \
\biggl(\frac{(\log n)^{d-1}}{n}\biggl)^{(r-1/p + 1/q)}.
\end{equation}
\end{corollary}

\begin{proof}
The upper bound of  \eqref{[r_n,p<q]} already is in Corollary \ref{corollary[r_n<,p<q]}. 
To prove the lower bound we compare the sampling width with the well known linear width which is defined by
\begin{equation} \nonumber 
\lambda_n(\Urp,L_q) 
\ := \ \inf_{\Lambda_n} \  \sup_{f \in W} \, \|f - \Lambda_n(f)\|_q,
\end{equation}
where the infimum is taken over all linear operators $\Lambda_n$ of rank $n$ in the normed space $L_q$.  The lower bound follows from the  inequality
$
r_n(\Urp,L_q)  
\ \ge \
\lambda_n(\Urp,L_q)  
$
and the inequality 
\begin{equation} \label{r_n>}
\lambda_n(\Urp,L_q) 
\ \gg \
\biggl(\frac{(\log n)^{d-1}}{n}\biggl)^{(r-1/p+1/q)}
\end{equation}
proven in \cite{Ga87} (see also \cite{Ga96}).
\end{proof}

Corollary \ref{corollary[r_n,p<q]} has been proven in \cite{ByDuUl14} for the case $2=p < q \le \infty$.

\medskip
\noindent
{\bf Final remarks.}
 All the results in this paper can be in a natural way extended to the Sobolev space 
$W^{\br}_p$ and class  $U^{\br}_p$ of nonuniform mixed smoothness $\br$ with 
$r=r_1 = \cdots = r_\nu < r_{\nu + 1} \le r_{\nu + 2} \le \cdots \le r_d$ by using the same methods and techniques. 
In particular, the direct and inverse Littlewood-Paley-type theorems of B-spline sampling representation for the space  $W^{\br}_p$ hold true, and in the results on asymptotic orders, upper and lower bounds of $r^s_n(U^{\br}_p,L_q)$ and $r_n(U^{\br}_p,L_q)$ the number $d$ is replaced by $\nu$.

The direct and inverse theorems of B-spline sampling representation  for Sobolev spaces of mixed smoothness in 
Section~\ref{Littlewood-Paley-type theorems}
can be also easily extended to Triebel-Lizorkin spaces of mixed smoothness.  With the extended theory of trigonometric sampling representation to Triebel-Lizorkin spaces the authors of \cite{ByUl15} are able to strengthen the results for Sobolev spaces and remove the smoothness restriction $r > \max(\frac{1}{p},\frac{1}{2})$ in the Sobolev context. 

\bigskip
\noindent
{\bf Acknowledgments.}  This work is funded by Vietnam National Foundation for Science and Technology Development (NAFOSTED) under  Grant No. A part of this work was done when the author was working as a research professor at the Vietnam Institute for Advanced Study in Mathematics (VIASM). He  would like to thank  the VIASM  for providing a fruitful research environment and working condition. The author would like to thank Glenn Byrenheid and Tino Ullrich for giving opportunity to read the manuscript~\cite{ByUl16}. He thanks Glenn Byrenheid, Vladimir Temlyakov and Tino Ullrich for useful discussions.

\end{document}